\documentclass[11pt,a4paper]{article}
\usepackage[english]{babel}
  \usepackage{amsmath,amsfonts,amssymb,relsize,amsthm}
  \usepackage{dsfont}
 \usepackage[left=2.1cm, right=2.1cm]{geometry}

  \usepackage{hyperref} 

 \usepackage[active]{srcltx} 
 \usepackage{color,soul} 

\newtheorem{theorem}{Theorem}[section]

\newtheorem{lemma}[theorem]{Lemma}
\newtheorem{proposition}[theorem]{Proposition}

\theoremstyle{definition}

\newtheorem{remark}{Remark}

\numberwithin{equation}{section}

\def\O{{\Omega}}
\def\o{{\omega}}
\def\eps{{\varepsilon}}
\def\a{{\mathcal{A}}}

\def\I{{\mathcal{I}}}
\def\j{{\mathcal{J}}}

\def\m{{\mathcal{M}}}

\def\r{{\mathcal{R}}}
\def\s{{\mathcal{S}}}
\def\E{{\mathcal{E}}}

\def\L{{\mathcal{L}}}
\def\M{{\mathcal{M}}}

\def\R{{\mathbb{R}}}

\def\N{{\mathbb{N}}}

\newcommand{\ope}[1]{\E[{#1}]}

\newcommand{\lb}[1]{\L_{_{#1}}}
\newcommand{\oplb}[2]{\L_{_{#2}}[{#1}]}
\newcommand{\oplm}[2]{\L_{_{\sigma,{#2}}}[{#1}]}

\newcommand{\mb}[1]{\M_{_{#1}}}
\newcommand{\opmb}[2]{\M_{_{#2}}[{#1}]}

\newcommand{\opmem}[2]{\M_{_{\sigma,{#2}}}[{#1}]}
\newcommand{\nlp}[3]{\|{#1}\|_{L^{#2}{(#3)}}}

\newcommand{\nlto}[1]{\|{#1}\|_{2}}
\newcommand{\nli}[2]{\|{#1}\|_{L^{\infty}{(#2)}}}
\newcommand{\flap}[2]{\Delta^{#2}{#1}}
\newcommand{\transposee}[1]{{\vphantom{#1}}^{\mathit t}{#1}}

\newtheorem{claim}[theorem]{Claim}

\newcommand{\fdem}{\vskip 0.2 pt \hfill $\square$  }
\def\ds{\displaystyle}

\newenvironment{formula}[1]{\begin{equation}\label{eq:#1}}
                       {\end{equation}\noindent}

\def\Fi#1{\begin{formula}{#1}}
\def\Ff{\end{formula}\noindent}

\setstcolor{red}

\title{On the definition and the properties of the principal eigenvalue of some nonlocal operators
}
\author{Henri Berestycki \thanks{
CAMS - \'Ecole des Hautes \'Etudes en Sciences Sociales,
190-198 avenue de France, 75013, Paris, France, {\itshape email:
}\ttfamily hb@ehess.fr} , J\'er\^ome Coville \thanks{
UR 546 Biostatistique et Processus  Spatiaux, INRA, Domaine St Paul Site Agroparc, F-84000 Avignon, France, {\itshape email:
}\ttfamily jerome.coville@avignon.fr} , Hoang-Hung Vo \thanks{  Institute for Mathematical Sciences, Renmin University of China, 59 Zhongguancun, Haidian, Beijing, 100872, China.  {\itshape present email:
}\ttfamily vhhungkhtn@gmail.com}}

\begin{document}

\maketitle
\begin{abstract}
In this article  we study some spectral properties of the linear operator $\lb{\O}+a$ defined on the space $C(\bar\O)$ by :
$$ \oplb{\varphi}{\O} +a\varphi:=\int_{\O}K(x,y)\varphi(y)\,dy+a(x)\varphi(x)
$$

 where $\O\subset \R^N$ is  a domain, possibly unbounded, $a$ is a continuous bounded function and 
 $K$ is a continuous, non negative kernel satisfying an integrability condition.

We focus our analysis on  the properties of the generalized principal eigenvalue $\lambda_p(\lb{\O}+a)$ defined by 
$$\lambda_p(\lb{\O}+a):= \sup\{\lambda \in \R \,|\, \exists \varphi \in C(\bar \O), \varphi>0,  \textit{ such that }\; \oplb{\varphi}{\O}+a\varphi +\lambda\varphi \le 0 \; \text{ in }\;\O\}. $$

 We  establish some new properties of this generalized principal eigenvalue $\lambda_p$. Namely,  we prove the equivalence of different definitions of the principal eigenvalue.  We also study the behaviour of  $\lambda_p(\lb{\O}+a)$ with respect to  some scaling of $K$.

For kernels $K$ of the type, $K(x,y)=J(x-y)$ with $J$ a  compactly supported probability density, we also establish some asymptotic properties of $\lambda_{p} \left(\lb{\sigma,m,\O} -\frac{1}{\sigma^m}+a\right)$  where $\lb{\sigma,m,\O}$ is defined by  
$\ds{\oplb{\varphi}{\sigma,2,\O}:=\frac{1}{\sigma^{2+N}}\int_{\O}J\left(\frac{x-y}{\sigma}\right)\varphi(y)\, dy}$. In particular, we prove that $$\lim_{\sigma\to 0}\lambda_p\left(\L_{\sigma,2,\O}-\frac{1}{\sigma^{2}}+a\right)=\lambda_1\left(\frac{D_2(J)}{2N}\Delta +a\right),$$
where $D_2(J):=\int_{\R^N}J(z)|z|^2\,dz$ and $\lambda_1$ denotes the Dirichlet principal eigenvalue of the elliptic operator. In addition,  we obtain some convergence results for the corresponding eigenfunction $\varphi_{p,\sigma}$.
\end{abstract}

\tableofcontents

\section{Introduction}

The principal eigenvalue of an operator is a fundamental notion in modern analysis. In particular, this notion is widely used in PDE's literature  and is at the source of many profound results especially in the study of elliptic semi linear problems. 
 For example, the principal eigenvalue  is  used to characterise the stability of equilibrium of a reaction-diffusion equation enabling the definition of persistence criteria \cite{Cantrell1989,Cantrell1991a,Cantrell1998,Berestycki2005,Englander2004,Kawasaki1997,Pinsky1996}.  It is also an important tool in the characterisation  of  maximum principle properties satisfies by elliptic operators \cite{Berestycki2015,Berestycki1994} and to describe  continuous  semi-groups that preserve an order \cite{Akian2014,Donsker1975,Lemmens2012}. It is further used in obtaining  Liouville type results for elliptic semi-linear equations  \cite{Berestycki2006,Berestycki2007}.

  In this article we are interested in such notion for  linear operators $\lb{\O}+a$ defined on the space of continuous functions $C(\bar \O)$ by :  
$$\oplb{\varphi}{\O}+a\varphi:=\int_{\O}K(x,y)\varphi(y)\,dy+a(x)\varphi(x)$$
where $\O\subset \R^N$ is a domain, possibly unbounded,  $a$ is a continuous bounded function and $K$ is a  non negative kernel satisfying an integrability condition. The precise assumptions on $\O,K$ and $a$ will be given later on.
    
To our knowledge, for most of  positive operators, the principal eigenvalue is a notion related to the existence of an eigen-pair, namely an eigenvalue associated with a positive eigen-element. For the operator $\lb{\O}+a$, when the function $a$ is not constant, for any real $\lambda$, neither $\lb{\O} +a+\lambda$ nor its inverse  are  compact operators.  Moreover, as noticed in \cite{Coville2010,Coville2013,Kao2010,Shen2010}, the operator $\lb{\O}+a$ may not have any eigenvalues in the space $L^p(\O)$ or $C(\bar \O)$. For such operator, the existence of an eigenvalue associated with a positive eigenvector  is then not guaranteed. Studying quantities that can be used as  surrogates of a principal eigenvalue and establishing their most important properties are therefore of great interest for such operators.

In this perspective, we are interested in the properties of the following quantity:
\begin{equation}\label{bcv-pev-def}
\lambda_p(\lb{\O}+a):= \sup\{\lambda \in \R \,|\, \exists \varphi \in C(\bar \O), \varphi>0,  \textit{ such that }\; \oplb{\varphi}{\O}+a(x)\varphi +\lambda\varphi \le 0 \; \text{ in }\;\O\},
\end{equation}
which can be  expressed equivalently by the $\sup \inf$ formula:
\begin{equation}
\lambda_p(\lb{\O}+a)=\sup_{\stackrel{\varphi \in C(\bar\O)}{\varphi >0}}\inf_{x\in \O}\left(-\frac{\oplb{\varphi}{\O}(x)+a(x)\varphi(x)}{\varphi(x)}\right).
\end{equation}

This  number  was originally introduced in the Perron-Frobenius Theory to characterise the eigenvalues of an irreducible positive matrix \cite{Collatz1942,Wielandt1950}. Namely, for a positive irreducible matrix $A,$  the eigenvalue $\lambda_1(A)$ associated with a positive eigenvector can be characterised as follows: 
\begin{equation}\label{bcv-pev-cw}
\lambda_p(A):=\sup_{\stackrel{x\in \R^N}{x>0}}\inf_{i\in\{1,\ldots,N\}}\left(-\frac{(Ax)_i}{x_i}\right)=\lambda_1(A)=\inf_{\stackrel{x\in \R^N}{x\ge 0, x\neq 0}}\sup_{i\in\{1,\ldots,N\}}\left(-\frac{(Ax)_i}{x_i}\right)=:\lambda'_p(A),
\end{equation} 
 also known as  the Collatz-Wieldandt characterisation.

Numerous  generalisation of these types of characterisation exist in the literature. Generalisations of the characterisation of the principal eigenvalue  by variants of the Collatz-Wielandt characterisation (i.e. \eqref{bcv-pev-cw}) were first obtained  for positive compact operators in $L^p(\O)$  \cite{Karlin1959,Karlin1964,Schaefer1984} and later for general positive operators that posses an eigen-pair \cite{Friedl1990}.

 In parallel with the generalisation of the Perron-Frobenius Theory, several inf sup formulas have been  developed to characterise the spectral properties of elliptic operators satisfying a maximum principle,  see  the fundamental works of Donsker and Varadhan \cite{Donsker1975},  Nussbaum, Pinchover \cite{Nussbaum1992}, Berestycki, Nirenberg, Varadhan \cite{Berestycki1994} and  Pinsky \cite{Pinsky1995,Pinsky1995a}. 
 In particular, for an elliptic operator defined in a bounded domain $\O\subset \R^N$  and  with bounded continuous coefficients,   $\E:=a_{ij}(x)\partial_{ij}+b_i(x)\partial_i +c(x)$, several notions of principal eigenvalue  have been introduced. On one hand, Donsker and Varadhan \cite{Donsker1975} have introduced  a quantity $\lambda_V(\E)$,  called  principal eigenvalue of $\E$, that satisfies
 \begin{align*}\lambda_V(\E)&:=\inf_{\stackrel{\varphi\in dom(\E)}{\varphi> 0}}\sup_{x\in\O}\left(-\frac{\ope{\varphi}(x)}{\varphi(x)}\right)\\
 &=\inf \{\lambda \in \R \, |\, \exists \, \varphi \in dom(\E), \varphi > 0 \;\text{ such that }\; \ope{\varphi}(x) + \lambda\varphi(x)\ge 0\; \text{ in }\; \O\}, 
 \end{align*} 
  where  $dom(\E)\subset C(\bar \O)$ denotes the domain of definition of $\E$.
  On the other hand Berestycki, Nirenberg and Varadhan \cite{Berestycki1994} have introduced  $\lambda_1(\E)$ defined by: 
 \begin{align*}\label{bcv-eq-lp-ellip}
 \lambda_1 (\E) &:= \sup \{\lambda \in \R \, |\, \exists \, \varphi \in W^{2,N}(\O), \varphi > 0 \;\text{ such that }\; \ope{\varphi}(x) + \lambda\varphi(x)\le 0\; \text{ in }\; \O\},\\
 &=\sup_{\stackrel{\varphi\in W^{2,N}(\O)}{\varphi>0}}\inf_{x\in \O}\left(-\frac{\ope{\varphi}(x)}{\varphi(x)}\right)
 \end{align*} 
 as another possible definition for the principal eigenvalue of $\E$. When $\O$ is a smooth bounded domain and $\E$ has smooth coefficients, both notions coincide (i.e. $\lambda_V(\E)=\lambda_1(\E)$).  The equivalence of this two notions has been recently extended for more general elliptic operators, in particular the equivalence holds true in any bounded domains $\O$ and in any domains when $\E$ is an elliptic self-adjoint operator with bounded coefficients \cite{Berestycki2015}. It is worth mentioning that the quantity $\lambda_V(\E)$ was originally introduced by Donsker and Varadhan \cite{Donsker1975} to obtain the following variational characterisation of $\lambda_1(\E)$ in a bounded domain:
 $$\lambda_{1}(E)= \sup_{d\mu \in \mathbb{P}(\O)} \inf_{\stackrel{\varphi\in dom(\E)}{\varphi> 0}}\int_{\O}\left(-\frac{\ope{\varphi}(x)}{\varphi(x)}\right)d\mu(x),$$    
 where $\mathbb{P}(\O)$ is the set of all probability measure on $\O$. Such characterisation is still valid when $\O$ is unbounded, see   Nussbaum and Pinchover \cite{Nussbaum1992}.

Lately, the search of Liouville type results for semilinear elliptic equations in unbounded domains \cite{Berestycki2006,Rossi2009} and the characterisation of spreading speed \cite{Berestycki2012,Nadin2010} have stimulated the studies of the properties  of $\lambda_1(\E)$ and several other notions of principal eigenvalue have emerged.  For instance, several new notions of principal eigenvalue have been introduced   for general elliptic operators defined on (limit or almost) periodic media \cite{Berestycki2007,Berestycki2006,Nadin2009b,Rossi2009}.  
For the interested reader, we refer  to  \cite{Berestycki2015}, for a review and a comparison of the different  notions of principal eigenvalue for an elliptic operator defined in a unbounded domain.

For the operator $\lb{\O}+a$, much less is known and  only partial results have been obtained when $\O$ is bounded \cite{Coville2010,Coville2015,Donsker1975,Garcia-Melian2009a,Ignat2012,Kao2010} or in a periodic media \cite{Coville2008b,Coville2013,Shen2010,Shen2012}. More precisely, $\lambda_p(\lb{\O}+a)$ has been compared to one of the following definitions : 
$$
\lambda_p'(\lb{\O} +a):=\inf\left\{\lambda\in\R\, |\, \exists \varphi \in C(\O)\cap L^{\infty}(\O), \varphi\ge_{\not\equiv} 0,\; \text{ s. t. }\;\oplb{\varphi}{\O}(x) +(a(x) +\lambda) \varphi(x)\ge 0 \; \text{ in }\; \O \right\}$$
or when $\lb{\O}+a$ is a self-adjoint operator:
\begin{align*}
\lambda_v(\lb{\O} +a)&:=\inf_{\varphi \in L^2(\O),\varphi \not\equiv 0} -\frac{ \langle \oplb{\varphi}{\O} +a\varphi,\varphi\rangle}{\nlp{\varphi}{2}{\O}^2},\\
&=\inf_{\varphi \in L^2(\O),\varphi \not\equiv 0} \frac{\frac{1}{2}\iint_{\O\times\O}K(x,y)[\varphi(x)-\varphi(y)]^2\,dxdy -\int_{\O}\left[a(x)+\int_{\O}K(x,y)\,dy\right]\varphi^2(x)\,dx}{\nlp{\varphi}{2}{\O}^2},
\end{align*}
where $ \langle,\rangle$ denotes the scalar product of $L^2(\O)$.
For $\O\subset \R^N$ a bounded domain  and for particular kernels $K$, an equality similar to $\lambda_V(\E)=\lambda_1(\E)$ has been obtained in \cite{Coville2010}, provided that $K\in C(\bar \O\times \bar\O)$  satisfies some non-degeneracy conditions. The author shows that 
\begin{equation}
\lambda_p(\lb{\O}+a)=\lambda'_p(\lb{\O}+a). \label{bcv-equa-lplp'}
\end{equation}

In a periodic media, an extension of this equality was  obtained in  \cite{Coville2008b,Coville2013} for kernels $K$ of the form $K(x,y):=J(x-y)$ with $J$ a symmetric positive continuous density of probability. In such case, they prove that
\begin{equation}
\lambda_p(\lb{\R^N}+a)=\lambda'_p(\lb{\R^N}+a)=\lambda_v(\lb{\R^N}+a). \label{bcv-equa-lplp'lv}
\end{equation}

In this paper, we pursue the works begun in \cite{Coville2010,Coville2013,Coville2015} by one of the present authors and we investigate more closely the properties of $\lambda_p(\lb{\O}+a)$.  Namely, we first look whether $\lambda_p(\lb{\O}+a)$ can be characterised by other notions of principal eigenvalue and under which conditions on $\O,K$ and $a$ the equality  \eqref{bcv-equa-lplp'} or \eqref{bcv-equa-lplp'lv} holds true. 
In particular, we  introduce a new notion of principal eigenvalue, $\lambda_p^{''}(\lb{\O}+a)$, defined by :
\begin{equation*}\label{def-lp''}
\lambda_p''(\lb{\O} +a):=\inf\left\{\lambda\in\R\, |\, \exists \varphi \in C_c(\O), \varphi\ge_{\not\equiv} 0,\; \text{ such that }\;\oplb{\varphi}{\O}(x) +(a(x) +\lambda) \varphi(x)\ge 0 \; \text{ in }\; \O \right\}
\end{equation*}
and we compare this new quantity with  $\lambda_p,\lambda_p'$ and $\lambda_v$.

Another natural question  is to obtain a clear picture on the dependence of $\lambda_p$ with respect to all the parameters involved. If the behaviour  of $\lambda_p( \lb{\O}+a)$ with respect to $a$ or $\O$ can be exhibited directly from the definition,  the impact of scalings of the kernel is usually unknown  and has been largely ignored in the literature except in some specific situations involving   particular nonlocal dispersal operators  defined  in  a bounded domain   \cite{Andreu2008,Cortazar2008,Kao2010, Shen2015}.

For a particular  type of $K$ and $a$,  we establish  the asymptotic properties  of $\lambda_p$ with respect to some scaling parameter.  
More precisely, let $K(x,y)=J(x-y)$  and let us denote $J_\sigma(z):=\frac{1}{\sigma^N}J\left (\frac{z}{\sigma}\right)$. When $J$ is a  non negative function of unit mass,  we study the properties of the principal eigenvalue of the operator $\ds{\L_{\sigma,m,\O}-\frac{1}{\sigma^m}+a}$, where  
the operator $\L_{\sigma,m,\O}$ is  defined  by:
\begin{align*}
&\oplm{\varphi}{m,\O}:=\frac{1}{\sigma^{m}}\int_{\O}J_\sigma(x-y)\varphi(y)\,dy.
\end{align*}
In this situation, the operator $\ds{\lb{\sigma,m,\Omega}-\frac{1}{\sigma^m}}$ refers to a  nonlocal version of the standard diffusion operator with a homogeneous Dirichlet boundary condition. 
Such type of  operators has appeared recently in the literature to model a population that have a constrained dispersal \cite{Berestycki2014,Fife1996,Hutson2003,Kao2010,Shen2015}.  In this context,  the pre-factor $\frac{1}{\sigma^m}$ is interpreted as a frequency at which the events of dispersal occur.
For $m\in [0,2]$ and a large class of $J$, we obtain the asymptotic limits of $\ds{\lambda_p\left(\L_{\sigma,m,\O}-\frac{1}{\sigma^m}+a\right)}$ as $\sigma\to 0$ and as $\sigma\to +\infty$.

 \subsection{Motivations: nonlocal reaction diffusion equation}
Our interest in studying the properties of $\lambda_p(\lb{\O}+a)$ stems from the recent studies of populations having a long range dispersal strategy \cite{Coville2010,Coville2015, Berestycki2014,Kao2010,Shen2010}. For such a population, a commonly used model that integrates such long range dispersal  is the following {\em nonlocal reaction diffusion equation} (\cite{Fife1996,Grinfeld2005,Hutson2003,Lutscher2005,Turchin1998}): 
\begin{equation}
 \partial_t u(t,x)=\int_{\O}J(x-y)u(t,y)\,dy - u(t,x)\int_{\O}J(y-x)\,dy   + f(x,u(t,x)) \quad \text{ in }\quad \R^+\times \O. \label{bcv-eq-dyn}
 \end{equation}
 In this context, $u(t,x)$ is  the density of the considered population, $J$ is a dispersal kernel and $f(x,s)$ is a KPP type non-linearity describing the growth rate of the population.  
When $\O$ is a bounded domain  \cite{Bates2007,Coville2010,Coville2015,Garcia-Melian2009,Kao2010,Shen2012}, an optimal  persistence criteria  has been obtained  using  the sign of $\lambda_p(\mb{\O}+\partial_uf(x,0))$, where $\mb{\O}$ stands for the operator:
$$\opmb{\varphi}{\O}:= \int_{\O}J(x-y) \varphi(y)\,dy -j(x)\varphi(x),$$
where $j(x):=\int_{\O}J(y-x)\,dy$.\\
In such model, a population will persists if and only if $\lambda_p(\mb{\O}+\partial_uf(x,0))< 0$. We can easily check that $\lambda_p(\mb{\O}+\partial_uf(x,0))=\lambda_p(\lb{\O}-j(x)+\partial_uf(x,0))$.

When $\O=\R^N$ and in periodic media, adapted versions of $\lambda_p$  have been recently used to define an optimal persitence criteria  \cite{Coville2008b,Coville2013,Shen2010,Shen2015}.  The  extension of such type of persistence criteria for  more general environments  is currently investigated by ourself \cite{Berestycki2014} by means of our findings on the properties of $\lambda_p$.

The understanding of the effect of a dispersal process conditioned by a \textit{dispersal budget} is another important question. The  idea introduced by Hutson, Martinez, Mischaikow and Vickers \cite{Hutson2003}, is simple and consists in introducing a cost function related to the amount of energy an individual has to use to produce offspring, that   jumps on a long range.  When a long  range of dispersal  is privileged,   the energy consumed to disperse an individual  is large  and so very few offsprings are dispersed. On the contrary, when the population chooses to disperse on  a short range, few energy is used  and  a large amount of the offsprings is  dispersed.     
In $\R^N,$ to understand the impact of a \textit{dispersal budget} on the range of dispersal,  we are led to consider the family of dispersal operator :
\begin{equation*}
\opmb{\varphi}{\sigma,m}(x):=\frac{1}{\sigma^{m}}\left(J_\sigma\star \varphi(x) -\varphi(x)\right),  
\end{equation*}
where $J_\sigma(z):=\frac{1}{\sigma^N}J\left(\frac{z}{\sigma}\right)$ is the standard scaling of the probability density $J$. For such family, the study of the dependence of $\lambda_p(\m_{\sigma,m}+a)$ with respect to $\sigma$ and $m$ is a first step to analyse  the impact of the range of the dispersal $\sigma$  on the persistence of the population. In particular the asymptotic limits $\sigma \to +\infty$ and $\sigma  \to 0$ are of primary interest.

\subsection{Assumptions and Main Results}
Let us now state the precise assumptions we are making on the domain $\O$, the kernel $K$ and the function $a$.
Here, throughout the paper, $\O\subset \R^N$ is a  domain (open connected set of $\R^N$) and for $a$ and $K$ we assume the following:  
\begin{equation} \label{hyp1}
a\in C(\bar \O)\cap L^{\infty}(\O),
\end{equation}
and $K$ is a non-negative Caratheodory function, that is $K\ge 0$ and,

\begin{equation}\label{hyp2}
  \forall\, x\in \O \; K(x,\cdot)\text{ is measurable,}\;   K(\cdot,y) \text{ is uniformly continuous for almost every }\; y\in \O.
\end{equation}

\noindent For our analysis, we  also require that $K$ satisfies the following non-degeneracy condition: 

\textit{There exist positive constants $\; r_0\ge r_1>0, C_0\ge c_0>0\;$  such that  $K$ satisfies:
\begin{equation}
C_0\mathds{1}_{_{\O\cap B_{r_0}(x)}}(y)\ge K(x,y)\ge c_0\mathds{1}_{_{\O\cap B_{r_1}(x)}}(y) \quad \text{for all }\; x,y\in \O, \label{hyp3}
\end{equation}
where $\mathds{1}_A$ denotes the characteristic function of the set $A\subset \R^N$ and $B_{r}(x)$ is the ball centred at $x$ of radius $r$.}
These conditions are satisfied for example for kernels like $K(x,y)=J\left(\frac{x-y}{g(y)h(x)}\right)$ with $h$ and $g$ positive and bounded in $\O$ and $J\in C(\R^N), J\ge 0$, a  compactly supported function such that $J(0)>0$. Note that when $\O$ is bounded, any  kernel $K \in C(\bar \O\times\bar \O)$ which is positive on the diagonal, satisfies all theses assumptions. 
Under this assumptions, we can check that the operator $\lb{\O}+a$ is  continuous in $C(\bar \O)$,\cite{Krasnoselʹskii1976}.

Let us now state our main results. We start by investigating the case of a bounded domain $\O$. In this situation, we prove that $\lambda_p,\lambda_p'$ and $\lambda_p''$ represent the same quantity. Namely, we show the following

\begin{theorem}\label{bcv-pev-thm1}
Let $\O\subset \R^N$ be a bounded domain and assume that $K$ and $a$ satisfy \eqref{hyp1} -- \eqref{hyp3}. Then, the following equality holds : 
$$\lambda_p(\lb{\O}+a)=\lambda_p'(\lb{\O}+a)=\lambda_p''(\lb{\O}+a).$$
In addition, if $K$ is symmetric, then 
$$\lambda_p(\lb{\O}+a)=\lambda_v(\lb{\O}+a).$$
\end{theorem}

When $\O$ is an  unbounded domain, the equivalence  of $\lambda_p, \lambda_p'$ and $\lambda_p''$ is not clear for general kernels.
Namely, let consider $\O=\R$, $K(x,y)=J(x-y)$ with $J$ a density of probability with a compact support and such that $\int_{\R}J(z)z\,dz>0$. For the operator  $\lb{\R}$, which corresponds to the standard convolution by $J$,   by using  $e^{\lambda x}$ and constants as test functions, we can easily check that $\ds{\lambda_p^{'}(\lb{\R})\le-1< -\min_{\lambda>0}\int_{\R}J(z)e^{-\lambda z}\,dz \le \lambda_p(\lb{\R})}$.  
However some inequalities remain true in general and the equivalence of the three notions holds for self-adjoint operators. More precisely, we prove  here the following
\begin{theorem}\label{bcv-pev-thm2}
Let $\O\subset \R^N$ be an unbounded domain and assume that $K$ and $a$ satisfy \eqref{hyp1} -- \eqref{hyp3}. Then the following inequalities hold 
$$\lambda_p'(\lb{\O}+a)\le\lambda_p''(\lb{\O}+a)\le\lambda_p(\lb{\O}+a).$$
When $K$ is symmetric and such that  $p(x):=\int_{\O}K(x,y)\,dy \in L^{\infty}(\O)$ then  the following equality holds : 
$$  \lambda_v(\lb{\O}+a)= \lambda_p'(\lb{\O}+a)=\lambda_p''(\lb{\O}+a)=\lambda_p(\lb{\O}+a).$$
\end{theorem}

Another striking property of $\lambda_p$ refers to the invariance of $\lambda_p$  under a particular scaling of the kernel $K$. More precisely, we show

\begin{proposition}\label{bcv-pev-prop-scal-eq}
Let $\O\subset \R^N$ be a domain and assume that  $a$ and $K$ satisfy \eqref{hyp1} -- \eqref{hyp3}. For all $\sigma>0$, let  
$\O_{\sigma}:=\sigma\O,$ $a_\sigma(x):=a\left(\frac{x}{\sigma}\right)$ and 
$$\oplb{\varphi}{\sigma,\O_\sigma}(x):=\frac{1}{\sigma^N}\int_{\O_\sigma}K\left(\frac{x}{\sigma},\frac{y}{\sigma}\right)\varphi(y)\,dy.$$
 Then for all $\sigma>0$, one has $$\lambda_p(\lb{\O}+a)=\lambda_p(\lb{\sigma,\O_\sigma}+a_\sigma).$$

\end{proposition}

Observe that no condition on the domain is imposed. Therefore, the invariance of $\lambda_p$ is still valid for  $\O=\R^N$. In this  case, since $\R^N$ is invariant under the scaling, we get   
$$\lambda_p(\lb{\R^N}+a)=\lambda_p(\lb{\sigma,\R^N}+a_\sigma).$$  

Next, for particular type of kernel $K$,  we investigate  the behaviour of $\lambda_p$ with respect of some scaling parameter. 
More precisely, let $K(x,y)=J(x-y)$  and let  $J_\sigma(z):=\frac{1}{\sigma^N}J\left (\frac{z}{\sigma}\right)$. We consider the following operator 
\begin{align*}
&\oplm{\varphi}{m,\O}:=\frac{1}{\sigma^{m}}\int_{\O}J_\sigma(x-y)\varphi(y)\,dy.
\end{align*}
 For $J$ is a  non negative function of unit mass,  we study the asymptotic  properties of the principal eigenvalue of the operator $\ds{\L_{\sigma,m,\O}-\frac{1}{\sigma^m}+a}$  when $\sigma \to 0$ and $\sigma \to +\infty$.

To simplify the presentation of our results, let us introduce the following notation. We denote by $\mb{\sigma,m,\O}$, the following operator:

\begin{equation}
\opmb{\varphi}{\sigma,m,\O}(x):=\frac{1}{\sigma^{m}}\left(\frac{1}{\sigma^{N}}\int_{\O}J\left(\frac{x-y}{\sigma}\right)\varphi(y)\,dy -\varphi(x)\right).\label{bcv-pev-eq-def-scalop}  
\end{equation}

For any domains $\O$,  we obtain the limits of $\lambda_p(\mb{\sigma,m,\O}+a)$ when $\sigma$  tends either to zero or to $+\infty$.  Let us denote the second moment of $J$ by 
 $$D_2(J):= \int_{\R^N} J(z)|z|^2\,dz,$$
the following statement describes the limiting behaviour of $\lambda_p(\mb{\sigma,m,\O}+a)$:

\begin{theorem}\label{bcv-pev-thm4}
 Let $\O$ be a  domain  and assume  that $J$ and $a$  satisfy \eqref{hyp1} -- \eqref{hyp3}. Assume further that $J$ is even and of unit mass. 
  Then, we have the following asymptotic behaviour:
   \begin{itemize}
\item When $0<m\le 2, \qquad\lim_{\sigma\to +\infty}\lambda_p(\mb{\sigma,m,\O}+a)=-\sup_{\O}a$
\item When $m=0, \qquad\lim_{\sigma\to +\infty}\lambda_p(\mb{\sigma,0,\O}+a)=1-\sup_{\O}a$.

In addition, when $\O=\R^N$ and if $a$ is  symmetric ($a(x)=a(-x)$ for all $x$) and the map $t\to a(tx)$ is non increasing for all $x,t>0$ then $\lambda_p(\mb{\sigma,0,\R^N}+a)$ is monotone non decreasing with respect to $\sigma$.
\item When $0\le m <2, \qquad\lim_{\sigma\to 0}\lambda_p(\mb{\sigma,m,\O}+a)=-\sup_{\O}a$ 
\item When $m=2$ and $a\in C^{0,\alpha}(\O)$ for some $\alpha >0$, then 

$$\lim_{\sigma\to 0}\lambda_p(\mb{\sigma,2,\O}+a)=\lambda_1\left(\frac{D_2(J)}{2N}\Delta +a\right)$$ 
and $$\lambda_1\left(\frac{D_2(J)}{2N}\Delta +a\right):=\inf_{\varphi \in H^1_0(\O),\varphi\not \equiv 0} \frac{D_2(J)}{2N}\frac{\int_{\O}|\nabla\varphi|^2(x)\,dx}{\nlto{\varphi}^2} -\frac{\int_{\O}a(x)\varphi^2(x)\,dx}{\nlto{\varphi}^2}. $$
\end{itemize}

 \end{theorem}

Note that the results hold for any  domains $\O$, so the results holds true in particular for $\O=\R^N$. Having established the asymptotic limits of the principal eigenvalue $\lambda_p(\m_{\sigma,2,\O}+a)$, it is natural to ask whether  similar results hold for the corresponding eigenfunction $\varphi_{\sigma,p}$ when it exists.
In this direction, we prove that for  $m=2$,  such convergence does occur :

\begin{theorem}\label{bcv-pev-thm5}
 Let $\O$ be any domain  and assume  that $J$ and $a$  satisfy \eqref{hyp1} -- \eqref{hyp3}. Assume further that $J$ is even and of unit mass. 
  Then there exists $\sigma_0$ such that for all $\sigma \le \sigma_0$, there exists a positive principal eigenfunction $\varphi_{p,\sigma}$ associated to $\lambda_p(\m_{\sigma,2,\O}+a)$. In addition, when $\varphi_{p,\sigma}\in L^2(\O)$ for all $\sigma\le\sigma_0$, we have 
  $$\varphi_{p,\sigma}\to \varphi_1 \quad \text{ in } \quad L_{loc}^2(\O), $$
where $\varphi_1\in H_0^1(\O)$ is a positive principle eigenfunction associated to $\lambda_1\left(\frac{D_2(J)}{2N}\Delta +a\right)$.
 \end{theorem}
\begin{remark}
When $\O$ is bounded, then the condition $\varphi_{p,\sigma}\in L^2$ is always satisfied. Moreover, in this situation, the  above limits $\varphi_{p,\sigma} \to \varphi_1 $ as $\sigma \to 0$ holds  in  $L^2(\O)$ instead of $L^2_{loc}(\O)$.
\end{remark}

 \medskip

\subsection{Comments and straightforward generalisation}

First, we can notice that the quantity $\lambda_V$ defined by Donsker and Varadhan \cite{Donsker1975} for elliptic operators can also be defined for the operator $\lb{\O}+a$ and  is equivalent to the quantity $\lambda'_p$. 
The equality \eqref{bcv-equa-lplp'} can then be seen as the nonlocal version of the equality $\lambda_1=\lambda_V$ where $\lambda_1$ is the notion introduced by Berestycki-Nirenberg-Varadhan \cite{Berestycki1994}.

 Next, we would like to emphasize, that unlike the classical elliptic  operators,  due to the lack of a regularising effect of the operator $\lb{\O}+a$,  the quantity $\lambda_p(\lb{\O}+a)$ may not be an eigenvalue, i.e. the spectral problem
$$\oplb{\varphi}{\O}(x)+a(x)\varphi(x) + \lambda \varphi(x) =0 \quad \text{ in }  \quad \O,  $$     
  may not have a solution  in  spaces of functions like $L^p(\O), C(\O)$\cite{Coville2008b,Coville2013a,Donsker1975,Kao2010}. 
 As a consequence, even in bounded domains, the  relations between $\lambda_p$, $\lambda_p',\lambda_p''$ and $\lambda_v$ are   quite delicate to obtain.  Another difficulty inherent to the study of nonlocal operators  in unbounded domains concerns the lack  of natural \textit{ a priori } estimates for the  positive eigenfunction thus making  standard approximation schemes difficult to use in most case.

Lastly, we make  some additional comments on the assumptions we have used on the dispersal kernel  $K$.
The non-degeneracy assumption \eqref{hyp3} we are using, is related to the existence of Local Uniform Estimates \cite{Cornea1994,Cornea1995} (Harnack type estimates) for a positive solution of a nonlocal equation:
\begin{equation}\label{bcv-pev-eq-lin-har}
\oplb{\varphi}{\O}+b(x)\varphi=0 \quad \text{ in }\quad \O.
\end{equation} 
Such type of  estimates is a key tool in our analysis, in particular  in unbounded domains, where we use it to obtain  fundamental properties of the principal eigenvalue $\lambda_p(\lb{\O}+a)$,  such as the limit:
$$\lambda_p(\lb{\O}+a)=\lim_{n\to \infty} \lambda_p( \lb{\O_n}+a),$$
where $\O_n$ is a sequence of set converging to $\O$.
As observed in  \cite{Coville2012},  some local uniform estimates can also be obtained for some particular kernels $K$ which does not satisfies the non-degeneracy condition \eqref{hyp3}. For example,  for  kernels of the form $K(x,y)=\frac{1}{g^N(y)}J\left(\frac{x-y}{g(y)}\right)$ with $J$ satisfying \eqref{hyp2}  and \eqref{hyp3} and $g\ge 0$ a bounded function such that $\{x|g(x)=0\}$ is a bounded set and with  Lebesgue measure zero,  some  local uniform estimates can be derived for positive solutions of \eqref{bcv-pev-eq-lin-har}. As a consequence, the Theorems \ref{bcv-pev-thm1} and \ref{bcv-pev-thm2} hold true for such kernels. 
We have also observed that the  condition \eqref{hyp3} can be slightly be relaxed and the Theorems \ref{bcv-pev-thm1} and \ref{bcv-pev-thm2} hold true   for kernels $K$ such that, for some positive integer $p$, the kernel $K_p$ defined  recursively by :
\begin{align*}
&K_1(x,y):=K(x,y),\\
&K_{n+1}(x,y):=\int_{\O} K_{n}(x,z)K_1(z,y)\,dz \quad \text{for } \quad n\ge 1,
\end{align*}
satisfies the non-degeneracy condition \eqref{hyp3}.
 
For a convolution operator, i.e. $K(x,y):=J(x-y)$, this last condition is optimal. It is related to a geometric property of the convex hull of $\{y\in\R^N| J(y)>0\}$: 
\smallskip

\textit{$K_p$ satisfies \eqref{hyp3} for some $p\in \N$   if and only if the convex hull of $\{y\in\R^N| J(y)>0\}$ contains $0$.}
\smallskip

Note that if a relaxed  assumption on the lower bound  of  the  non-degeneracy condition satisfied by $K$ appears  simple to find, the condition on the support of $K$ seems quite tricky to relax. To tackle this problem,  it is tempting to investigate the spectrum of  linear operators involving  the Fractional  
 Laplacian, $\Delta^\alpha$:
\begin{align*}
&\flap{\varphi}{\alpha}:=C_{N,\alpha} P.V.\left(\int_{\O}\frac{\varphi(y)-\varphi(x)}{|x-y|^{N+2\alpha}}dy \right),
&\varphi\equiv 0 \quad \text{ in }\quad \R^N\setminus \O  
\end{align*}
That is, to look for the properties of the principal eigenvalue of the spectral problem:
\begin{equation}\label{bcv-pev-eq-fraclap}
\flap{\varphi}{\alpha} +(a+\lambda)\varphi=0 \quad \text{ in }\quad \O. 
\end{equation}
As for  elliptic operators and $\lb{\O}+a$, analogues of $\lambda_1,\lambda_1'$ and $\lambda_0$ can be defined for $\Delta^{\alpha} +a$ and the relations between all  possible definitions can be investigated. When $\O$ is bounded or $a$ is periodic,  the different definitions are equivalent \cite{Berestycki2011}.
However,  in  the situations considered in \cite{Berestycki2011} the operator $\Delta^{\alpha} +a$ has a compact resolvent enabling the use of the Krein Rutmann Theory. Thus, the corresponding $\lambda_p$  is associated with  a positive eigenfunction, rendering  the relations much more simpler to obtain.  Moreover, in this analysis,  the regularity of the principal eigenfunction and a Harnack type inequality \cite{Cabre2014,Caffarelli2009,Tan2011} for some non negative  solution of \eqref{bcv-pev-eq-fraclap} are  again  the key ingredients in the proofs yielding to  the  inequality
$$\lambda_p'(\Delta^{\alpha}+a,\O)\le \lambda_p(\Delta^{\alpha}+a,\O)$$  for any smooth domain $\O$.

Such Harnack type inequalities are not known for  operators $\lb{\O} +a$ involving a continuous kernel $K$ with unbounded support. Furthermore, it seems that  most of the tools used to establish these Harnack estimates in the case of the Fractional Laplacian \cite{Cabre2014, Tan2011} do not apply  when we consider an operator  $\lb{\O} +a$. Thus, obtaining the inequality
$$\lambda_p'(\lb{\O}+a)\le \lambda_p(\lb{\O}+a)$$
with a more general  kernel  requires  a deeper understanding of Harnack type estimates and/or the development of new analytical tools for such type of nonlocal operators.

Nevertheless, in this direction and in dimension one, for some kernels with unbounded support, we  could obtain some inequalities between  the different notions of principal eigenvalue. Namely,

\begin{proposition}\label{bcv-pev-thm3}
Assume $N=1$ and let $\O\subset \R$ be a unbounded domain.  Assume that $K$ and $a$ satisfy \eqref{hyp1}--\eqref{hyp2}. Assume further that $K$ is symmetric   and there exists $C>0$ and $\alpha>\frac{3}{2}$ such that  $K(x,y)\le C(1+|x-y|)^{-\alpha}$. Then we have 
\begin{align*}
&\lambda_p(\lb{\O}+a)\le \lambda_v(\lb{\O}+a)\le \lambda_p'(\lb{\O} +a)\le  \lambda_p''(\lb{\O}+a).
\end{align*}
\end{proposition}

\bigskip
\bigskip

\textit{Outline of the paper:} 
The paper is organised as follows. In Section \ref{bcv-section-pre}, we recall some  known results
and properties of the principal eigenvalue $\lambda_p(\lb{\O}+a)$. The relations between the different definitions of the principal eigenvalue, $\lambda_p, \lambda_p'$, $\lambda_p''$ and $\lambda_v$ (Theorems \ref{bcv-pev-thm1}, \ref{bcv-pev-thm2}  and Proposition  \ref{bcv-pev-thm3}) are proved in  Section  \ref{bcv-section-pev}.
Finally, in Section \ref{bcv-section-scal} we derive the asymptotic behaviour of $\lambda_p$ with respect to the different scalings of $K$ (Proposition \ref{bcv-pev-prop-scal-eq} and Theorems  \ref{bcv-pev-thm4} and \ref{bcv-pev-thm5} ).

\subsection{Notations}
To simplify the presentation of the proofs, we introduce some notations and various linear operator that we will use throughout this paper:
\begin{itemize}
\item $B_R(x_0)$  denotes the standard ball of radius $R$ centred at the point $x_0$
\item $\mathds{1}_R$ will always refer to the characteristic function of the ball $B_R(0)$.
\item $\s(\R^N)$ denotes the Schwartz space,\cite{Brezis2010} 
\item $C(\O)$ denotes the space of continuous function in $\O$, 
\item $C_c(\O)$ denotes the  space of continuous function with compact support in  $\O$.
\item For a positive integrable function $J\in \s(\R^N)$, the constant $\int_{\R^N}J(z)|z|^2\,dz$ will refer to 
$$\int_{\R^N}J(z)|z|^2\,dz:=\int_{\R^N}J(z)\left(\sum_{i=1}^Nz_i^2\right)\,dz$$
\item For a bounded set $\o\subset\R^N$, $|\o|$ will denotes its Lebesgue measure
\item For two $L^2$ functions $\varphi,\psi$, $\langle\varphi,\psi \rangle$  denotes the $L^2$ scalar product of $\psi$ and $\varphi$ 
\item For $J \in L^1(\R^N)$, $J_\sigma(z):= \frac{1}{\sigma^N}J\left(\frac{z}{\sigma}\right)$
\item We denote by $\lb{\sigma,m,\O}$ the continuous linear operator 
\begin{equation}\label{bcv-def-opl}
\begin{array}{rccl}
\lb{\sigma,m,\O}:&C(\bar \O)&\to& C(\bar \O)\\
&\varphi&\mapsto& \ds{\frac{1}{\sigma^m}\int_{\O}J_\sigma(x-y)u(y)\,dy},
\end{array}
\end{equation}
where $\O\subset \R^N$.
\item We denote by $\mb{\sigma,m,\O}$  the operator $\ds{\mb{\sigma,m,\O}:=\lb{\sigma,m,\O} - \frac{1}{\sigma^m} }$
\end{itemize}


\section{Preliminaries}\label{bcv-section-pre}

 In this section, we recall some standard results on the principal eigenvalue of the operator $\lb{\O}+a$. 
Since the early work \cite{Donsker1975} on the variational formulation of the principal eigenvalue, an intrinsic difficulty  related to the study of these quantities comes  from the possible  non-existence of a positive continuous eigenfunction associated to the definition of $\lambda_p,\lambda_p',\lambda_p''$ or to  $\lambda_v$. This means that there is not always a positive continuous eigenfunction associated to $\lambda_p, \lambda_p',\lambda_p''$ or $\lambda_v$. A simple illustration of this fact  can be found in  \cite{Coville2010, Coville2013a}.
Recently, some progress have been made in the understanding  of $\lambda_p$. In particular, some flexible criteria  have been found  to guarantee  the existence of a positive continuous eigenfunction \cite{Coville2010,Kao2010,Shen2012}.
More precisely,
\begin{theorem}[Sufficient condition \cite{Coville2010}] \label{bcv-pev-th-crit1}
Let  $\O\subset \R^N$ be a domain, $a\in C(\O)\cap L^{\infty}(\O)$ and $K\in C(\bar\O\times\bar\O)$ non negative, satisfying the condition \eqref{hyp3}. Let us denote $\nu:=\sup_{\bar \O} a$ and assume further that the function   $a$ satisfies $\frac{1}{\nu -a}\not \in L^1(\O_0)$ for some bounded domain $\O_0\subset\bar \O$. Then there exists a principal eigen-pair $(\lambda_p,\varphi_p)$ solution of
$$\oplb{\varphi}{\O}(x)+(a(x) +\lambda)\varphi(x)=0 \quad \text{ in }\quad \O.$$
 Moreover, $\varphi_p\in C(\bar \O)$,  $\varphi_p>0$ and we have the following estimate
$$ -\nu' <\lambda_p<-\nu, $$ where $\ds{\nu':=\sup_{x\in\O } \left[a(x)+ \int_{\O}K(x,y)\,dy\right]}$.
\end{theorem}

This criteria is almost optimal, in the sense that we can construct example of operator $\lb{\O}+a$ with $\O$ bounded and $a$ such that $\frac{1}{\nu -a}\in L^1 (\O )$ and where $\lambda_p(\lb{\O}+a)$ is not an eigenvalue in $C(\bar \O)$, see \cite{Coville2010,Kao2010,Shen2012}. 

When $\O$ is bounded, sharper results have been recently derived in \cite{Coville2013a} where it is proved that  $\lambda_p(\lb{\O}+a)$ is always an eigenvalue in the Banach space of positive measure, that is, we can always find a positive measure $d\mu_p$ that is solution in the sense of measure of 

\begin{equation}
\oplb{d\mu_p}{\O}(x)+a(x)d\mu_p(x)+\lambda_p d\mu_p(x)=0. \label{bcv-pev-eq-measure}
\end{equation}

In addition, we have the following characterisation of $\lambda_p$:

\begin{theorem}[\cite{Coville2013,Coville2013a}]\label{bcv-pev-th-crit2}
$\lambda_p(\lb{\O}+a)$ is an eigenvalue in $C(\bar \O)$ if and only if   $\ds{\lambda_p(\lb {\O}+a)<-\sup_{x\in\O}a(x)}$.
\end{theorem}

We refer to  \cite{Coville2013a} for a more complete description of the positive solution associated to $\lambda_p$ when the domain $\O$ is bounded.

  Now, we recall some properties of  $\lambda_p$ that we  constantly use throughout this paper:
 \begin{proposition}\label{bcv-prop-pev}
\begin{itemize}
\item[(i)] Assume $\O_1\subset\O_2$, then for the two operators $\lb{\O_1}+a$ and $\lb{\O_2}+a$

respectively defined on $C(\O_1)$ and $C(\O_2)$, we have :
 $$
\lambda_p(\lb{\O_1}+a)\ge \lambda_p(\lb{\O_2}+a).
$$
\item[(ii)]For a fixed $\O$  and assume that $a_1(x)\ge a_2(x)$, for all $x \in \O$. Then  
$$
\lambda_p(\lb{\O}+a_2)\ge\lambda_p(\lb{\O}+a_1).
$$

\item[(iii)] $\lambda_p(\lb{\O}+a)$ is Lipschitz continuous with respect to  $a$. More precisely,
$$|\lambda_p(\lb{\O}+a)- \lambda_p(\lb{\O}+b)|\le \|a-b\|_{\infty}$$

\item[(iv)]  The following estimate always holds
$$-\sup_{\O}\left(a(x)+\int_{\O}K(x,y)\,dy\right)\le \lambda_p(\lb{\O}+a)\le -\sup_{\O}a.$$
\end{itemize}
\end{proposition}
We refer to \cite{Coville2010,Coville2015} for the proofs of $(i)-(iv)$.

Lastly, we prove some limit  behaviour of $\lambda_p(\lb{\O}+a)$ with respect to the domain $\O$. Namely, we show
\begin{lemma}\label{bcv-lem-lim}
Let $\O$ be a domain and assume that  $a$ and $K$ satisfy \eqref{hyp1}--\eqref{hyp3}. 
Let $(\O_n)_{n \in \N}$ be a sequence of subset of $\O$ so that $\lim_{n\to \infty}\O_n =\O$, $\O_n \subset \O_{n+1}$.
Then we have 
$$ \lim_{n\to \infty}\lambda_p(\lb{\O_n}+a)=\lambda_p(\lb{\O}+a)$$ 
\end{lemma}

\begin{proof}

By a straightforward application of the monotone properties of $\lambda_p$ with respect to the domain  ((i) of Proposition \ref{bcv-prop-pev}) we get the inequality 
\begin{equation}\label{bcv-eq-lim-lpn}
\lambda_p(\lb{\O}+a)\le \lim_{n\to \infty}\lambda_p(\lb{\O_n}+a).   
\end{equation}
To prove the equality, we argue by contradiction. So, let us assume  
\begin{equation}\label{bcv-eq-lim-lpn0}
\lambda_p(\lb{\O}+a)< \lim_{n\to \infty}\lambda_p(\lb{\O_n}+a),   
\end{equation}
and choose $\lambda \in \R$  such that 
\begin{equation}\label{bcv-eq-lim-lpn1}
\lambda_p(\lb{\O}+a)<\lambda< \lim_{n\to \infty}\lambda_p(\lb{\O_n}+a).
\end{equation}

We claim 
\begin{claim}\label{bcv-claim-testfct}
There exists $\varphi>0$, $\varphi\in C(\O)$ so that $(\lambda,\varphi)$ is an adequate test function. That is, $\varphi$ satisfies 
$$ \oplb{\varphi}{\O}(x)+(a(x)+\lambda)\varphi(x)\le 0 \quad \text{in}\quad \O.$$
\end{claim}
Assume for the moment that the above claim holds. By definition of $\lambda_p(\lb{\O}+a)$, we get a straightforward contradiction   
 $$ \lambda_p(\lb{\O}+a)< \lambda\le \lambda_p(\lb{\O}+a).$$
 
Hence, $$\lim_{n\to \infty}\lambda_p(\lb{\O_n}+a)=\lambda_p(\lb{\O}+a)$$
\end{proof}
Let us  now prove  Claim \ref{bcv-claim-testfct}
\begin{proof}[Proof of Claim \ref{bcv-claim-testfct}]
By definition of $\nu:=\sup_\O a$, there exists a sequence of points $(x_k)_{k\in\N}$ such that $x_k\in\O$ and 
$|a(x_k)-\nu|<\frac{1}{k}$.  By continuity of $a$, for each $k$, there exists $\eta_k>0$ such that 
$$B_{\eta_k}(x_k)\subset \O,\quad \text{ and }\quad \sup_{B_{\eta_k}(x_k)}|a-\nu|\le \frac{2}{k}.$$

Now, let $\chi_k$ be  the following    cut-off" functions :
 $\chi_k(x):=\chi\left(\frac{\|x_k-x\|}{\eps_k}\right)$ where $\eps_k>0$ is to be chosen later on and $\chi$ is a smooth function such that $0\le \chi \le 1$, $\chi(z) = 0$ for $|z| \ge 2$ and $\chi(z) = 1$ for $|z| \le 1$. 
Finally, let us consider the continuous functions $a_k(\cdot)$, defined by   $a_k(x):=\sup\{a,(\nu-\inf_{\O}a)\chi_k(x)+\inf_{\O}a\}$.
By taking a sequence $(\eps_k)_{k \in\N}$ so that $\eps_k\le \frac{\eta_k}{2}$, $\eps_k\to 0$,  we have 
$$a_k(x)=\begin{cases}
a \quad &\text{ for } \quad x\in \O\setminus B_{2\eps_k}(x_k)\\
\nu \quad &\text{ for } \quad x\in \O\cap B_{\eps_k}(x_k)
\end{cases}
$$
and therefore 
 $$\|a-a_k\|_{\infty}\le \sup_{B_{\eta_k}(x_k)}|\nu-a|\to 0 \quad \text{ as }\quad k \to \infty.$$

By construction, for $k$ large enough, say $k\ge k_0$ , we get for all $k\ge k_0$
$$\|a-a_k\|_{\infty}\le \inf\left\{\frac{|\lambda_p(\lb{\O} +a)-\lambda|}{2}\, , \frac{|\lim_{n\to \infty}\lambda_p(\lb{\O_n} +a)-\lambda|}{2}\right\}.$$
Since $\O_n\to \O$ when $n\to \infty$, there exists $n_0:=n(k_0)$ so that 

  $$B_{\eta_{k_0}}(x_{k_0})\subset \O_n \quad \text{for all}\quad n\ge n_0.$$

On the othre hand, from the Lipschitz continuity of $\lambda_p(\lb{\O}+a)$ with respect to $a$ ((iii) Proposition \ref{bcv-prop-pev}),   inequality \eqref{bcv-eq-lim-lpn1} yields     
  
\begin{equation}\label{bcv-eq-lim-lpn2}
\lambda_p(\lb{\O}+a_{k_0})<\lambda< \lim_{n\to \infty}\lambda_p(\lb{\O_n}+a_{k_0}). 
\end{equation}

Now, by construction we see that for   $n\ge n_0$, $\sup_{\O_{n}}a_{k_0}=\sup_{\O}a_{k_0}=\nu$ and since $a_{k_0}\equiv \nu$ in  $B_{\frac{\eps_{k_0}}{2}}(x_{k_0})$,  for all $n\ge n_0$ the function $\frac{1}{\nu-a_{k_0}}\not \in L^1_{loc}(\bar \O_n)$. Therefore, by Theorem \ref{bcv-pev-th-crit1}, for all $n\ge n_0$ there exists 
$\varphi_n \in C(\bar \O_n)$, $\varphi_n>0$ associated with $\lambda_p(\lb{\O_n}+a_{k_0})$. 
 
Moreover, since $x_{k_0}\in \bigcap_{n\ge n_0}\O_n$, for all $n\ge n_0$,   we can normalize $\varphi_n$ by 
$\varphi_n(x_{k_0})=1$. Recall that for all $n\ge n_0$, $\varphi_n$ satisfies
$$
\oplb{\varphi_n}{\O_n}(x)+(a_{k_0}(x)+\lambda_p(\lb{\O_n}+a_{k_0}(x)))\varphi_n(x)=0 \quad \text{in}\quad \O_n, 
$$
so from \eqref{bcv-eq-lim-lpn2}, it follows that $(\varphi_n,\lambda)$ satisfies
\begin{equation}\label{bcv-eq-lim-lpn3}
\oplb{\varphi_n}{\O_n}(x)+(a_{k_0}(x)+\lambda)\varphi_n(x)<\oplb{\varphi_n}{\O_n}(x)+(a_{k_0}(x)+\lambda_p(\lb{\O_n}+a_{k_0}))\varphi_n(x)=0 \quad \text{in}\quad \O_n. 
\end{equation}

Let us now define   $b_n(x):=-\lambda_{p}(\lb{\O_n}+a_{k_0}(x))-a_{k_0}(x)$, then for all $n\ge n_0,$ $\varphi_{n}$ satisfies  
\begin{equation}\label{bcv-eq-lim-lpn4}
  \oplb{\varphi_{n}}{\O_n}(x)=b_n(x)\varphi_{n}(x) \quad \text{in}\quad \O_n.
  \end{equation}
 By construction, for  $n\ge n_0$, we have $b_n(x)\ge-\lambda_{p}(\lb{\O_{n_0}}+a_{k_0}(x))-\nu>0 $.  Therefore, since $K$ satisfies the condition \eqref{hyp3}, the Harnack inequality (Theorem 1.4  in \cite{Coville2012}) applies to $\varphi_{n}$. Thus, for $n\ge n_0$ fixed and for any compact set $\o \subset \subset \O_n$ there exists a constant $C_n(\o)$ such that 
$$\varphi_{n}(x)\le C_n(\o)\varphi_{n}(y) \quad \forall \; x,y \in \o.$$

Moreover, the constant $C_n(\o)$ only depends  on $\delta_0<\frac{d(\o,\partial \O)}{4}$, $c_0$,  $\bigcup_{x\in \o}B_{\delta_0}(x)$  and $\inf_{\O_n}b_n$. Furthermore, this constant is   decreasing  with respect to  $\inf_{\O_n}b_n$.
Notice that for all $n\ge n_0$, the function $b_n(x)$ being uniformly bounded from below by a constant independent of $n$, the constant $C_n$ is bounded from above independently of $n$ by a constant $C(\o)$.  Thus, we have 
$$\varphi_{n}(x)\le C(\o)\varphi_{n}(y) \quad \forall \quad x,y \in \o.$$ 

From a standard argumentation, using the normalization $\varphi_{n}(x_{k_0})=1$, we deduce that  the sequence $(\varphi_{n})_{n\ge n_0}$  is uniformly bounded in $C_{loc}(\O)$ topology and is locally uniformly equicontinuous. Therefore, from a standard diagonal extraction argument, there exists a subsequence, still denoted $(\varphi_{n})_{n\ge  n_0}$, such that $(\varphi_{n})_{n\ge n_0}$ converges  locally uniformly  to  a continuous function  $\varphi$ which is nonnegative, non trivial function and satisfies $\varphi(x_{k_0})=1$.

Since $K$ satisfies the condition \eqref{hyp3}, we can pass to the limit in the Equation \eqref{bcv-eq-lim-lpn3} using the Lebesgue  monotone convergence theorem   and we  get 
$$ \oplb{\varphi}{\O}+(a_{k_0}(x)+\lambda)\varphi(x) \le 0 \quad \text{ in }\quad \O.$$
Hence, we have 
 $$ \oplb{\varphi}{\O}(x)+(a(x)+\lambda)\varphi(x) \le 0 \quad \text{ in }\quad \O,$$
 since $a\le a_{k_0}$.
\end{proof}


\section{Relation between $\lambda_p,\lambda_p', \lambda_p''$ and $\lambda_v$}
\label{bcv-section-pev}
In this section, we investigate the relations between the quantities $\lambda_p,\lambda_p',\lambda_p''$ and $\lambda_v$ and prove  Theorems \ref{bcv-pev-thm1} and \ref{bcv-pev-thm2}.

First, remark that, as  consequences of the definitions,  the monotone  and Lipschitz continuity properties satisfied by $\lambda_p$ ($(i)-(iii)$ of Proposition \eqref{bcv-prop-pev}) are still true for $\lambda_p'$ and $\lambda_v$.
We investigate now  the relation between $\lambda_p'$ and $\lambda_p$:  
\begin{lemma}\label{bcv-lem-lp'-le-lp}
Let $\O\subset \R^N$ be a domain and assume that $K$ and $a$ satisfy \eqref{hyp1}--\eqref{hyp3}. Then,  
 $$\lambda_p'(\lb{\O}+a)\le \lambda_p(\lb{\O}+a).$$
\end{lemma}

\begin{proof}

Observe that  to get  inequality $\lambda_p'(\lb{\O}+a)\le \lambda_p(\lb{\O}+a),$
 it is sufficient to show that for any $\delta>0$: 
 $$ \lambda_p'(\lb{\O}+a)\le \lambda_p(\lb{\O}+a)+\delta.$$

For $\delta>0$, let us  consider the operator $\lb{\O}+b_\delta$ where $b_\delta:=a+\lambda_p(\lb{\O}+a)+\delta$. We claim that 
\begin{claim}\label{bcv-cla-lp'-le-lp}
For all $\delta>0$, there exists $\varphi_\delta\in C_c(\O)$ such that $\varphi_\delta\ge 0$ and $\varphi_\delta$ satisfies
$$\oplb{\varphi_\delta}{\O}(x)+b_\delta(x)\varphi_\delta(x)\ge 0 \quad \text{in}\quad \O.$$
\end{claim}
By proving the claim, we prove the Lemma. Indeed, assume for the moment that the claim holds. Then, by construction,
$(\varphi_\delta,\lambda_p(\lb{\O}+a)+\delta)$ satisfies 
$$\oplb{\varphi_\delta}{\O}(x)+[a(x)+\lambda_p(\lb{\O}+a)+\delta]\varphi_\delta(x)\ge 0 \quad \text{in}\quad \O.$$
Thus, by definition of $\lambda_p'(\lb{\O}+a)$, we have 
$\lambda_p'(\lb{\O}+a)\le \lambda_p(\lb{\O}+a)+\delta$.
The constant $\delta$ being arbitrary, we get for all $\delta>0$: $$\lambda_p'(\lb{\O}+a)\le \lambda_p(\lb{\O}+a)+\delta.$$

\end{proof}

\begin{proof}[Proof of the Claim]

Let $\delta>0$ be fixed. By construction $\lambda_p(\lb{\O}+b_\delta)<0$, so by Lemma \ref{bcv-lem-lim}, there exists a bounded open set $\o$ such that  $\lambda_p(\lb{\o}+b_\delta)<0$.
For any $\eps>0$ small enough, by taking $\o$ larger if necessary,  arguing as in the proof of Claim \ref{bcv-claim-testfct}, we can find $b_\eps$ such that 

\begin{align*}
&\|b_\delta-b_\eps\|_{\infty,\o}=\|b_\delta-b_\eps\|_{\infty,\O}\le \eps,\\
&\lambda_p(\lb{\o}+b_\eps(x))+\eps<0,
\end{align*}
 and there is $\varphi_p \in C(\bar \o)$, $\varphi_p>0$ associated to  $\lambda_p(\lb{\o}+b_\eps(x))$. That is $\varphi_p$ satisfies
\begin{equation}\label{bcv-eq-approx-delta}
\oplb{\varphi_p}{\o}(x) +b_\eps(x)\varphi_p(x) =-\lambda_p(\lb{\o}+b_\eps(x))\varphi_p(x) \quad\text{ in }\quad\o.
\end{equation} 
 Without loss of generality,  assume that   $\varphi_p\le 1$.

Let $\nu$ denotes the maximum of $b_\eps$ in $\bar\o$, then  by Proposition \ref{bcv-prop-pev}, there exists $\tau>0$ such that $$-\lambda_p(\lb{\o}+b_\eps(x))-\eps -\nu\ge \tau>0.$$
 Moreover, since $\varphi_p$ satisfies \eqref{bcv-eq-approx-delta}, there exists $d_0>0$ so that $\inf_{\o}\varphi_p\ge d_0$. \\
Let us choose $\o'\subset \subset \o$ such that $$|\o\setminus \o'|\le\frac{d_0\inf\{\tau,-\lambda_p(\lb{\o}+b_\eps)-\eps\}}{2\|K\|_\infty},$$
where for a set A,  $|A|$ denotes the Lebesgue measure of $A$. 

Since $\bar \o'\subset \subset \o$ and $\partial \o$ are two disjoint closed sets, by the Urysohn's Lemma there exists a continuous function $\eta$ such that $0\le \eta \le 1$, $\eta = 1$ in $\o'$, $\eta = 0$ in $\partial \o$.
Consider now $\varphi_p\eta$ and let us compute $\oplb{\varphi_p\eta}{\o}+b_\delta\varphi_p\eta.$ Then, we have 
\begin{align*}
\oplb{\varphi_p\eta}{\o}+b_\delta\varphi_p\eta &\ge-\lambda_p(\lb{\o}+b_\eps)\varphi_p -\|K\||\o\setminus \o'|-b_\eps\varphi_p(1-\eta)-(b_\eps-b_\delta)\varphi_p\eta,\\
&\ge-(\lambda_p(\lb{\o}+b_\eps)+\|b_\delta-b_\eps\|_{\infty,\o})\varphi_p -\frac{d_0\inf\{\tau,-\lambda_p(\lb{\o}+b_\eps)-\eps\}}{2}-b_\eps(x)\varphi_p(1-\eta),\\
&\ge-(\lambda_p(\lb{\o}+b_\eps)+\eps)\varphi_p -\frac{d_0\inf\{\tau,-\lambda_p(\lb{\o}+b_\eps)-\eps\}}{2}-\max\{\nu,0\}\varphi_p,\\
&\ge-(\lambda_p(\lb{\o}+b_\eps)+\eps+\max\{\nu,0\})\varphi_p -\frac{d_0\inf\{\tau,-\lambda_p(\lb{\o}+b_\eps)-\eps\}}{2}.
\end{align*} 
Since $-\lambda_p(\lb{\o}+b_\eps)-\eps>0$ and $-\lambda_p(\lb{\o}+b_\eps)-\eps-\nu\ge \tau>0$,   from the above inequality, we infer that
\begin{align*}
\oplb{\varphi_p\eta}{\o}+b_\delta\varphi_p\eta &\ge-(\lambda_p(\lb{\o}+b_\eps)+\eps+\max\{\nu,0\})d_0 -\frac{d_0\inf\{\tau,-\lambda_p(\lb{\o}+b_\eps)-\eps\}}{2},\\
&\ge \frac{d_0\inf\{\tau,-\lambda_p(\lb{\o}+b_\eps)-\eps\}}{2}\ge 0.
\end{align*}
By construction, we have $\varphi_p\eta \in C(\o)$ satisfying  
\begin{align*}
&\oplb{\varphi_p\eta}{\o} +b_\delta\varphi_p\eta \ge 0 \quad\text{ in } \quad  \o, \\
&\varphi_p\eta= 0 \quad\text{ on } \quad \partial \o.
\end{align*}
By extending $\varphi_p\eta$ by $0$ outside $\o$ and denoting $\varphi_\delta$ this extension, we get

\begin{align*}
&\oplb{\varphi_\delta}{\O}(x) +b_\delta(x)\varphi_\delta(x)= \oplb{\varphi_\delta}{\o}(x) +b_\delta(x)\varphi(x) \ge 0 \quad\text{ in } \quad  \o,\\
&\oplb{\varphi_\delta}{\O}(x) +b_\delta(x)\varphi_\delta(x)= \oplb{\varphi_\delta}{\o}(x) \ge 0 \quad\text{ in } \quad  \O\setminus \o.
 \end{align*}
Hence, $\varphi_\delta \ge 0, \varphi \in C_c(\O)$ is the desired test function.
\end{proof}

\begin{remark}\label{bcv-pev-rem-lp'-le-lp}
The assumption \eqref{hyp3} on  $K$ is only needed to reduce the problem on  unbounded domains to  problem  on  bounded domains. In addition, the above construction shows that the inequality is still valid if we replace $\lambda_p'$ by
$\lambda_p''(\lb{\O}+a)$. Thus we have for any domain $\O$,
$$\lambda_p''(\lb{\O}+a)\le \lambda_p(\lb{\O}+a).$$
 \end{remark}

\subsection{The bounded case: }
Assume for the moment that $\O$ is  a bounded domain and let us show that  the three definitions $\lambda_p,\lambda_p'$ and $\lambda_p''$ are equivalent and if in addition $K$ is symmetric,  $\lambda_v$ is equivalent to $\lambda_p$. We start by the case $\lambda_p'=\lambda_p$. Namely, we show 
\begin{lemma}\label{bcv-lem-lp'-eq-lp}
Let $\O$ be a bounded domain of $\R^N$ and assume that $a$ and $K$ satisfy \eqref{hyp1}--\eqref{hyp3}.  Then, 
$$\lambda_p(\lb{\O}+a)= \lambda_p'(\lb{\O}+a).$$
In addition, when $\lb{\O}+a$ is self adjoined, we have 
 $$\lambda_p(\lb{\O}+a)=\lambda_v(\lb{\O}+a).$$
\end{lemma}

The proof of Theorem \ref{bcv-pev-thm1} is  a straightforward consequence of the above Lemma.  Indeed, by Remark \ref{bcv-pev-rem-lp'-le-lp} and the definition of $\lambda_p''$ we have 
$$\lambda_p'(\lb{\O}+a) \le \lambda_p''(\lb{\O}+a)\le \lambda_p(\lb{\O}+a).$$
Thus, from the above Lemma we get  
$$ \lambda_p(\lb{\O}+a)=\lambda_p'(\lb{\O}+a)\le   \lambda_p''(\lb{\O}+a)\le \lambda_p(\lb{\O}+a)=\lambda_p'(\lb{\O}+a)    .$$ 
\fdem

Let us now turn  to the proof of Lemma \ref{bcv-lem-lp'-eq-lp}
\begin{proof}[Proof of Lemma \ref{bcv-lem-lp'-eq-lp}]
By Lemma \ref{bcv-lem-lp'-le-lp}, we already have 
$$\lambda_p'(\lb{\O}+a)\le \lambda_p(\lb{\O}+a). $$
So, it remains to prove the converse inequality. Let us assume  by contradiction  that 
$$ \lambda_p'(\lb{\O}+a)< \lambda_p(\lb{\O}+a).$$
Pick now $\lambda \in (\lambda_p'(\lb{\O}+a), \lambda_p(\lb{\O}+a))$, then, by definition of $\lambda_p$ and $\lambda_p'$, there exists  $\varphi$ and $\psi$ non negative  continuous  functions such that 
\begin{align*}
\oplb{\varphi}{\O}(x)+(a(x)+\lambda)\varphi(x)\le 0 \quad \text{ in }\quad \O,\\
 \oplb{\psi}{\O}(x)+(a(x)+\lambda)\psi(x)\ge 0\quad \text{ in }\quad \O.
\end{align*}
Moreover, $\varphi>0$ in $\bar \O$. By taking $\lambda$ smaller if necessary, we can assume that $\varphi$ satisfies
$$\oplb{\varphi}{\O}(x)+(a(x)+\lambda)\varphi(x)< 0 \quad \text{ in }\quad \O. $$
A direct computation yields 
$$\int_{\O}K(x,y)\varphi(y)\left(\frac{\psi(y)}{\varphi(y)} -\frac{\psi(x)}{\varphi(x)}\right)\,dy > 0.$$
Since $\frac{\psi}{\varphi}\in C(\bar\O)$, the function $\frac{\psi}{\varphi}$  achieves a maximum at some point $x_0\in \bar \O$, evidencing thus  the contradiction:
 $$0<\int_{\O}K(x_0,y)\varphi(y)\left(\frac{\psi(y)}{\varphi(y)} -\frac{\psi(x_0)}{\varphi(x_0)}\right)\,dy \le 0.$$
Thus, $$\lambda_p'(\lb{\O}+a)=\lambda_p(\lb{\O}+a). $$

In the self-adjoined case,  it is enough to prove that $$\lambda_p'(\lb{\O}+a)=\lambda_v(\lb{\O}+a).$$
From the definitions of $\lambda_p'$ and $\lambda_v$, we easily obtain that $\lambda_v \le \lambda_p'$. Indeed,
let  $\lambda >\lambda_p'(\lb{\O}+a)$, then by definition of $\lambda_p'$ there exists $\psi\ge 0$ such that $\psi \in C(\O)\cap L^{\infty}(\O)$ and 
  \begin{equation}\label{bcv-eq-l2sup}
  \oplb{\psi}{\O}(x)+(a(x)+\lambda)\psi(x)\ge 0\quad \text{ in }\quad \O.
  \end{equation}
 Since $\O$ is bounded and $\psi\in L^{\infty}(\O)$,    $\psi \in L^2(\O)$. So,   multiplying \eqref{bcv-eq-l2sup} by $-\psi$ and integrating over $\O$ we get 
 \begin{align*}
 &-\int_{\O}\int_{\O}K(x,y)\psi(x)\psi(y)\,dxdy -\int_{\O}a(x)\psi(x)^2\,dx \le \lambda \int_{\O}\psi^2(x)\,dx,\\
 & \frac{1}{2}\int_{\O}\int_{\O}K(x,y)\left(\psi(x)-\psi(y)\right)^2\,dxdy -\int_{\O}(a(x)+k(x))\psi(x)^2\,dx \le \lambda \int_{\O}\psi^2(x)\,dx,\\
  &\lambda_v(\lb{\O}+a)\int_{\O}\psi^2(x)\,dx\le   \lambda \int_{\O}\psi^2(x)\,dx.
 \end{align*}
 Therefore, $\lambda_v(\lb{\O}+a)\le \lambda_p'(\lb{\O}+a)$.\\
 Let us prove now the converse inequality. Again,  we argue by contradiction and let us assume that 
 \begin{equation}\label{bcv-eq-lv-le-lp'}
 \lambda_v(\lb{\O}+a)<\lambda_p'(\lb{\O}+a).
 \end{equation}

Observe first that by density of $C(\bar \O)$ in $L^2(\O)$, we  easily check that 
\begin{align*}
-\lambda_v(\lb{\O}+a)&=-\inf_{\varphi \in L^{2}(\O),\varphi\not \equiv 0} \frac{\frac{1}{2}\int_{\O}\int_{\O}K(x,y)(\varphi(x)-\varphi(y))^2\,dydx -\int_{\O}(a(x)+k(x))\varphi(x)^2\,dx}{\|\varphi\|^2_{L^2(\O)}},\\
&=-\inf_{\varphi \in L^{2}(\O),\varphi\not \equiv 0} \frac{-\int_{\O}\int_{\O}K(x,y)\varphi(x)\varphi(y)\,dydx -\int_{\O}a(x)\varphi(x)^2\,dx}{\|\varphi\|^2_{L^2(\O)}},\\
&=\sup_{\varphi \in L^{2}(\O),\varphi\not \equiv 0} \frac{\langle\oplb{\varphi}{\O}+a\varphi,\varphi\rangle}{\|\varphi\|^2_{L^2(\O)}},\\
&=\sup_{\varphi \in C(\bar\O),\varphi\not \equiv 0} \frac{\langle\oplb{\varphi}{\O}+a\varphi,\varphi\rangle}{\|\varphi\|^2_{L^2(\O)}}.
\end{align*}

 By (iv) of Proposition \ref{bcv-prop-pev}, since $\lambda_p'(\lb{\O}+a)=\lambda_p(\lb{\O}+a)$,   from \eqref{bcv-eq-lv-le-lp'} we infer that $\lambda_+ $ defined by
\begin{equation}\label{bcv-eq-sigma+} 
\lambda_+ = \sup_{\varphi \in C(\overline \Omega) }
\frac{\langle\oplb{\varphi}{\O} + a \varphi, \varphi\rangle
}{\int_\Omega \varphi^2}
\end{equation}
satisfies
\begin{align}
\label{bcv-eq-bound-sigma+} \lambda_+  >-\lambda_p(\lb{\O}+a)\ge \max_{\overline
\Omega} a.
\end{align}

Now, using the same arguments as in \cite{Coville2008b,Hutson2003}, we infer that
the supremum in \eqref{bcv-eq-sigma+} is achieved. Indeed, it is a standard fact
\cite{Brezis2010} that the spectrum of $\lb{\O} + a$ is at the left of
$\lambda_+$ and that there exists a sequence $\varphi_n \in
C(\overline \Omega)$ such that $\|\varphi_n \|_{L^2(\Omega)}=1$
and $ \| (\lb{\O} + a - \lambda_+) \varphi_n \|_{L^2(\Omega)} \to 0$
as $n\to+\infty$. By compactness of $\lb{\O}:L^2(\Omega) \to
C(\overline\Omega)$, for a subsequence, $\lim_{n\to+\infty}
\lb{\O}[\varphi_n] $ exists in $C(\overline \Omega)$. Then, using
\eqref{bcv-eq-bound-sigma+}, we see that $\varphi_n \to \varphi$ in
$L^2(\Omega)$ for some $\varphi $ and $(\lb{\O}+a)\varphi = \lambda_+
\varphi$. This equation implies $ \varphi \in C(\overline \Omega)$,
and $\lambda_+$ is an eigenvalue for the operator
$\lb{\O}+a$. Moreover, $\varphi\ge 0$, since  $\varphi^+$ is also a minimizer. Indeed, we have
 \begin{align*}
 \lambda_+ &=\frac{\int_{\O}[\oplb{\varphi}{\O}(x)+a(x)\varphi(x)]\varphi^+(x)\,dx}{\|\varphi^+\|^2_{L^2(\O)}},\\
 &=\frac{\int_{\O}[\oplb{\varphi^+}{\O}(x)+a(x)\varphi^+(x)]\varphi^+(x)\,dx}{\|\varphi^+\|^2_{L^2(\O)}} +\frac{\int_{\O}\int_{\O}K(x,y)\varphi^-(x)\varphi^+(y)\,dydx}{\|\varphi^+\|^2_{L^2(\O)}},\\
 &\le\frac{\int_{\O}[\oplb{\varphi^+}{\O}+a\varphi^+(x)]\varphi^+(x)\,dx}{\|\varphi^+\|^2_{L^2(\O)}}\le \lambda_+.
 \end{align*}

Thus, there exists a non-negative continuous $\varphi$ so that 
$$\oplb{\varphi}{\O}(x)+(a(x)+\lambda_v)\varphi(x)=0\quad \text{in}\quad \O.$$
Since $\lambda_v<\lambda_p$, we can  argue as above and get the desired contradiction.
Hence, $\lambda_v=\lambda_+=\lambda_p=\lambda_p'$. 

\end{proof}

\subsection{The unbounded case:}

Now let $\O$ be an unbounded domain. From Lemma \ref{bcv-lem-lp'-le-lp} and Remark \ref{bcv-pev-rem-lp'-le-lp}, we already know that 
$$\lambda_p'(\lb{\O}+a)\le \lambda_p''(\lb{\O} +a)\le \lambda_p(\lb{\O} +a).$$
To complete the proof of Theorem \ref{bcv-pev-thm2}, we are then left to prove that
$$\lambda_p'(\lb{\O}+a)= \lambda_p''(\lb{\O} +a)= \lambda_p(\lb{\O} +a)=\lambda_v(\lb{\O}+a),$$
 when $\lb{\O}+a$ is self-adjoined and the kernel $K$ is such that $p(x):= \int_{\O}K(x,y)\,dy$ is a bounded function in $\O$. To do so, we prove the following inequality :

\begin{lemma}\label{bcv-lem-lp-le-lpprime} 
Let $\O$ be an unbounded domain and assume  that $a$ and $K$ satisfies \eqref{hyp1}--\eqref{hyp3}. Assume further that $K$ is symmetric and  $p(x):=\int_{\O}K(x,y) \,dy \in L^{\infty}(\O)$. Then, we have  $$\lambda_p(\lb{\O}+a)\leq \liminf_{n\to +\infty}\lambda_v(\lb{\O_n}+a)\leq\lambda_p'(\lb{\O} +a),$$
where $\O_n:=(\O\cap B_n)_{n\in\N}$ and $B_n$  is the ball of radius $n$ centred at $0$.
\end{lemma}

Assume  that Lemma \ref{bcv-lem-lp-le-lpprime} holds  and let us end the proof of Theorem \ref{bcv-pev-thm2}.
\begin{proof}[ Proof of Theorem \ref{bcv-pev-thm2} :]~\\
From Lemma \ref{bcv-lem-lp'-le-lp} and \ref{bcv-lem-lp-le-lpprime}, we get the inequalities:  

\begin{align*}
&\lim_{n\to \infty}\lambda_v(\lb{\O_n}+a)\le \lambda_p'(\lb{\O}+a)\le \lambda_p''(\lb{\O}+a)\le \lambda_p(\lb{\O}+a),\\
&\lambda_p(\lb{\O}+a) \le \lim_{n\to \infty}\lambda_v(\lb{\O_n}+a)\le \lambda_p'(\lb{\O}+a)\le \lambda_p''(\lb{\O}+a),
\end{align*}
with $\O_n:=\O\cap B_n(0)$.
Therefore, $$ \lim_{n\to \infty}\lambda_v(\lb{\O_n}+a)= \lambda_p'(\lb{\O}+a)=\lambda_{p}''(\lb{\O}+a)= \lambda_p'(\lb{\O}+a)= \lambda_p(\lb{\O}+a).$$
It remains to prove that $\lambda_v(\lb{\O}+a)=\lambda_p(\lb{\O}+a)$.

By definition of $\lambda_p''(\lb{\O} +a)$,  we  check that 
$$\lambda_v(\lb{\O}+a)\le \lambda_p''(\lb{\O}+a)=\lambda_p(\lb{\O}+a).$$
On the other hand, by definition of $\lambda_v(\lb{\O} +a)$, for any $\delta>0$ there exists $\varphi_\delta \in L^{2}(\O)$ such that 

$$ \frac{\frac{1}{2}\iint_{\O\times\O}K(x,y)(\varphi_\delta(x)-\varphi_\delta(y))^2\,dydx-\int_{\O}(a(x)+p(x))\varphi_\delta^2(x)\,dx}{\nlp{\varphi_\delta}{2}{\O}^2} \le \lambda_v(\lb{\O}+a) +\delta. $$

Define
$$\I_R(\varphi_\delta):=\frac{\frac{1}{2}\iint_{\O_R\times \O_R}K(x,y)(\varphi_\delta(x)-\varphi_\delta(y))^2\,dydx - \int_{\O_R}(a(x)+p_R(x))\varphi_\delta^2(x)\,dx}{\nlp{\varphi_\delta}{2}{\O_R}^2}, $$
with $p_R(x):=\int_{\O_R}K(x,y)\,dy$.
Since $\lim_{R\to \infty}p_R(x)=p(x)$ for all $x\in \O$, $a\in L^{\infty}$ and $\varphi_\delta \in L^2(\O)$, by Lebesgue's monotone convergence Theorem  we get for $R$ large enough

$$-\int_{\O_R}(a(x)+p_R(x))\varphi_\delta^2(x)\,dx\le \delta \nlp{\varphi_\delta}{2}{\O_R}^2 -\int_{\O}(a(x)+p(x))\varphi_\delta^2(x)\,dx. $$ 

Thus, we have for $R$ large enough

\begin{align*}
\I_R(\varphi_\delta)&\le \frac{\frac{1}{2}\iint_{\O\times\O}K(x,y)(\varphi_\delta(x)-\varphi_\delta(y))^2\,dydx-\int_{\O}(a(x)+p(x))\varphi_\delta^2(x)\,dx}{\nlp{\varphi_\delta}{2}{\O}^2},\\
&\le \frac{\nlp{\varphi_\delta}{2}{\O}^2}{\nlp{\varphi_\delta}{2}{\O_R}^2}( \lambda_v(\lb{\O}+a) +\delta) +\delta,\\
&\le   \lambda_v(\lb{\O}+a) + C\delta,
\end{align*}
for some universal constant $C>0$.

By definition of $\lambda_v(\lb{\O_R}+a)$, we then get
$$ \lambda_v(\lb{\O_R}+a)\le \I_R(\varphi_\delta) \le \lambda_v(\lb{\O}+a) + C\delta  \quad \text{ for $R$ large enough}.$$
Therefore,
\begin{equation}\label{bcv-eq-lvR-le-lv}
\lim_{R\to \infty}  \lambda_v(\lb{\O_R}+a)\le \lambda_v(\lb{\O}+a) + C\delta.
\end{equation}
Since \eqref{bcv-eq-lvR-le-lv} holds true for any $\delta$, we get 
$$\lim_{R\to \infty}  \lambda_v(\lb{\O_R}+a)\le \lambda_v(\lb{\O}+a).$$
As a consequence, we obtain 

$$\lambda_p(\lb{\O}+a)= \lim_{n\to \infty}\lambda_v(\lb{\O_n}+a)\le \lambda_v(\lb{\O}+a)\le \lambda_p''(\lb{\O} +a)=\lambda_p(\lb{\O}+a),$$
which enforces $$\lambda_v(\lb{\O}+a)=\lambda_p(\lb{\O} +a).$$

\end{proof}

We can now turn  to the proof of Lemma \ref{bcv-lem-lp-le-lpprime}. But before proving this Lemma, we start by showing some technical Lemma in the spirit of Lemma 2.6 in \cite{Berestycki2011}. Namely, we prove  

\begin{lemma}\label{bcv-pev-lem-tech}
Assume $\O$ is unbounded and let $g \in L^{\infty}(\O)$ be a non negative function, then for any $R_0>0$, we have 
$$\lim_{R\to \infty}\frac{\int_{\O\cap (B_{R_0+R}\setminus B_R)}g}{\int_{\O\cap B_R}g}=0.$$
\end{lemma}

\begin{proof}
Without loss of generality, by extending $g$ by $0$ outside $\O$ we can assume that $\O=\R^N$. For any $R_0,R>0$ fixed, let us denote the annulus 
$C_{R_0,R}:=B_{R_0+R}\setminus B_{R}$. Assume by contradiction that 
$$\lim_{R\to \infty}\frac{\int_{C_{R_0,R}}g}{\int_{B_R}g}>0.$$
Then there exists $\eps>0$ and $R_\eps >1$ so that 
$$\forall\, R\ge R_\eps, \quad \frac{\int_{C_{R_0,R}}g}{\int_{B_R}g}\ge \eps.$$
Consider the sequence $(R_n)_{n\in \N}$ defined by $R_n:=R_\eps +nR_0$ and set $a_n:=\int_{C_{R_0,R_n}}g$.
For all $n$, we have $ C_{R_0,R_n}=B_{R_{n+1}}\setminus B_{R_n}$ and   
$$ B_{R_{n+1}}=B_{R_\eps} \cup \left(\bigcup_{k=0}^{n}C_{R_0,R_k}\right).$$ 
From the last inequality, for $n\ge 1$ we deduce that 
$$a_n\ge \eps\int_{B_{R_n}}g \ge \eps \sum_{k=0}^{n-1}a_k.$$
Arguing now as in \cite{Berestycki2011},by a recursive argument, the last inequality yields
\begin{equation}\label{bcv-pev-eq-lem-luca} 
\forall\, n\ge 1, \quad a_n\ge \eps a_0(1+\eps)^{n-1}.
\end{equation}
On the other hand, we have
$$a_n=\int_{C_{R_0,R_n}}g \le \|g\|_{\infty}|C_{R_0,R_n}| \le d_0 n^{N},  $$
with $d_0$ a positive constant, contradicting thus   \eqref{bcv-pev-eq-lem-luca}.

\end{proof}

We are now in a position  to  prove  Lemma \ref{bcv-lem-lp-le-lpprime}.

\begin{proof}[Proof of Lemma \ref{bcv-lem-lp-le-lpprime} :]

The proof follows some ideas developed in \cite{Berestycki2008,Berestycki2011,Coville2008b,Coville}. To simplify the presentation, let us call $\lambda_p=\lambda_p(\lb{\O}+a)$ and $\lambda_p'=\lambda_p'(\lb{\O}+a)$.

First recall that  for a bounded domain $\O$, we have 
$$ \lambda_p=\lambda_p'=\lambda_v.$$
Let $(B_n)_{n\in\N}$ be the increasing sequence of balls of radius $n$ centred at $0$ and let $\O_n:=\O\cap B_n$.
By monotonicity of $\lambda_p$ with respect to the domain,  we  have 
$$\lambda_p(\lb{\O}+a)\le \lambda_p(\lb{\O_n}+a)=\lambda_v(\lb{\O_n}+a)$$
Therefore 
$$\lambda_p(\lb{\O}+a)\le \liminf_{n\to\infty}\lambda_v(\lb{\O_n}+a).$$
Thanks to the last inequality,  we obtain the inequality  $\lambda_p(\lb{\O}+a)\le \lambda_p'(\lb{\O}+a)$ by proving that 
\begin{equation}\label{bcv-eq-liminf-le-lpprime}
\liminf_{n\to\infty}\lambda_v(\lb{\O_n}+a)\le \lambda_p'(\lb{\O}+a).
\end{equation}
To prove \eqref{bcv-eq-liminf-le-lpprime}, it is enough to  show that for any $\delta>0$  
\begin{equation}\label{bcv-eq-liminf-le-lpprime2}
\liminf_{n\to\infty}\lambda_v(\lb{\O_n}+a)\le \lambda_p'(\lb{\O}+a)+\delta.
\end{equation}
Let us fix $\delta>0$ and let us denote $\mu:=\lambda_p'(\lb{\O}+a)+\delta$.
By definition of $\lambda_p'(\lb{\O}+a)$ there exists a function $\varphi\in C(\O)\cap L^\infty(\O)$, $\varphi\ge 0$ satisfying
\begin{equation}\label{bcv-eq-liminf-le-lpprime3}
 \oplb{\varphi}{\O}(x)+a(x)\varphi(x)+\mu\varphi(x)\geq 0\quad    \text{in} \quad \O.
 \end{equation}
Without loss of generality, we can also assume that  $\|\varphi\|_{L^\infty(\O)}=1$.

Let $\mathds{1}_{\O_n}$ be the characteristic function of $\O_n=\O\cap B_n$ and let $w_n=\varphi\mathds{1}_{\O_n}$. 
By definition of $\lambda_v(\lb{\O_n}+a)$ and  since $w_n\in L^2(\O_n)$, we have 
\begin{equation}\label{bcv-eq-liminf-le-lpprime4}
\lambda_v(\lb{\O_n}+a)\nlp{w_n}{2}{\O_n}^2\le \int_{\O_n}\left(-\oplb{w_n}{\O_n}(x)-a(x)w_n(x)\right)w_n(x)\,dx.
\end{equation}
Since $\oplb{\varphi}{\O}w_n \in L^1(\O_n)$, from \eqref{bcv-eq-liminf-le-lpprime4} and by using  \eqref{bcv-eq-liminf-le-lpprime3} we get  
\begin{align*}
\lambda_v(\lb{\O_n}+a)\nlp{w_n}{2}{\O_n}^2&\le \int_{\O_n}\left(-\oplb{w_n}{\O_n}(x)-a(x)w_n(x)-\mu w_n+\mu w_n\right)w_n(x)\,dx,\\
&\le \mu\nlp{w_n}{2}{\O_n}^2+\int_{\O_n}\left(-\oplb{w_n}{\O_n}(x)+\oplb{\varphi}{\O}(x)\right)w_n(x)\,dx,\\
&\le \mu\nlp{w_n}{2}{\O_n}^2+\int_{\O_n}\left(\int_{\O\setminus \O_n}K(x,y)\varphi(y)\,dy\right)w_n(x)\,dx,\\
&\le \mu\nlp{w_n}{2}{\O_n}^2+I_n,
\end{align*}
where  $I_n$ denotes 
\begin{equation*} 
I_n:=\int_{\O_n}\left(\int_{\O\setminus \O_n}K(x,y)\varphi(y)\,dy\right)\varphi(x)\,dx.
\end{equation*}

Observe that we achieve  \eqref{bcv-eq-liminf-le-lpprime2} by proving 
\begin{equation}\label{bcv-eq-liminf-le-lpprime_In}
\liminf_{n\to\infty}\frac{I_n}{\nlp{\varphi}{2}{\O_n}^2}=0.
\end{equation}
Recall that $K$ satisfies \eqref{hyp3}, therefore there exists $C>0$ and $R_0>0$ such that \\$K(x,y)\leq C\mathds{1}_{R_0}(|x-y|))$. So, we get 
\begin{equation}\label{bcv-eq-liminf-le-lpprime5}
I_n\le \int_{\O_n}\left(\int_{\O\cap (B_{R_0+n}\setminus B_n)}K(x,y)\varphi(y)\,dy\right)w_n(x)\,dx.
\end{equation} 

By Fubini's Theorem, Jensen's inequality and  Cauchy-Schwarz's inequality, it follows that    
\begin{align*}
I_n&\leq \left(\int_{\O\cap(B_{R_0+n}\setminus B_n)}\varphi^2(y)\,dy\right)^{1/2}\left(\int_{\O\cap (B_{R_0+n}\setminus B_n)}\left(\int_{\O_n}K(x,y)\varphi(x)\,dx\right)^{2}\,dy\right)^{1/2},\\
&\leq \nlp{\varphi}{2}{\O\cap (B_{R_0+n}\setminus B_n)}\left(\int_{\O\cap(B_{R_0+n}\setminus B_n)}\left(\int_{ \O\cap B_n}K^2(x,y)\varphi^2(x)\,dx\right)\,dy\right)^{1/2},\\
&\leq \nlp{\varphi}{2}{\O\cap(B_{R_0+n}\setminus B_n)}\left(\int_{\O\cap B_n}\left(\int_{\O\cap B_{R_0+n}\setminus B_n)}K^2(x,y)\,dy\right)\varphi^2(x)\,dx\right)^{1/2}.
\end{align*}
Since $K$ and $p$ are bounded functions, we obtain 
\begin{equation}
I_n\leq \|K\|_{\infty}\|p\|_{\infty}\nlp{\varphi}{2}{\O \cap(B_{R_0+n}\setminus B_n}\nlp{\varphi}{2}{\O\cap B_n}.\label{bcv-eq-liminf-le-lpprime6}
\end{equation}
 Dividing \eqref{bcv-eq-liminf-le-lpprime6} by $\nlp{\varphi}{2}{\O_n}^2$,  we then get
 
 $$\frac{I_n}{\nlp{\varphi}{2}{\O_n}^2}\le C \frac{\nlp{\varphi}{2}{\O \cap (B_{R_0+n}\setminus B_n)}}{\nlp{\varphi}{2}{\O\cap B_n}}.$$
Thanks to Lemma \ref{bcv-pev-lem-tech}, the right hand side of the above inequality tends to $0$ as $n\to\infty$. 
Hence,  we get
\begin{equation}
\liminf_{n\to \infty}\lambda_v(\lb{\O_n}+a)\le \mu +\liminf_{n\to\infty}\frac{I_n}{\nlp{\varphi_n}{2}{\O_n}^2}=\lambda_p'(\lb{\O}+a)+\delta. \label{bcv-eq-liminf-le-lpprime9}
\end{equation}

Since  the above arguments holds true for any arbitrary $\delta>0$,   the Lemma is proved.
\end{proof}

\begin{remark}
In the above proof, since $w=\varphi \mathds{1}_{\O_n} \in L^{2}(\O)$ and $\oplb{w_n}{\O}=\oplb{w_n}{\O_n}$, the inequality 
\eqref{bcv-eq-liminf-le-lpprime4} is true with $\lambda_v(\lb{\O}+a)$ instead of $\lambda_v(\lb{\O_n}+a)$. Thus, we get immediately
$$\lambda_v(\lb{\O}+a)\le \lambda_p'(\lb{\O}+a).$$
\end{remark}

When $N=1$, the decay restriction imposed on the kernel can be weakened, see \cite{Coville}. In particular,   we have  

\begin{lemma}
Let $\O$ be an unbounded domain and assume  that $a$ and $K$ satisfy \eqref{hyp1}--\eqref{hyp2}. Assume further that $K$ is symmetric  and $K$ satisfies $0\le K(x,y)\leq C(1+|x-y|)^{-\alpha}$ for some $\alpha>\frac{3}{2}$. Then one has $$\lambda_p(\lb{\O}+a)\leq \liminf_{n\to +\infty}\lambda_v(\lb{\O_n}+a)\leq\lambda_p'(\lb{\O} +a),$$
where $\O_n:=\O\cap (-n,n)$.
\end{lemma}
\begin{proof}
By arguing as in the above proof, for any $\delta>0$ there exists $\varphi\in C(\O)\cap L^{\infty}(\O)$ such that 
$$\oplb{\varphi}{\O}+(a+\lambda_p(\lb{\O}+a)+\delta)\varphi(x) \ge 0 \quad \text{ in } \quad \O.$$
and   
$$
\lambda_v(\lb{\O_n}+a)\nlp{w_n}{2}{\O_n}^2\le \mu\nlp{w_n}{2}{\O_n}^2+I_n,
$$
where  $\mu:=\lambda_p(\lb{\O}+a)+\delta), w_n:=\varphi\mathds{1}_{(-n,n)}$ and $I_n$ denotes 
\begin{equation} \label{bcv-eq-liminf-le-lpprimeN1}
I_n:=\int_{\O_n}\left(\int_{\O\setminus \O_n}K(x,y)\varphi(y)\,dy\right)\varphi(x)\,dx.
\end{equation}

As above, we  end our proof  by showing 
\begin{equation}\label{bcv-eq-limIn}
\liminf_{n\to\infty}\frac{I_n}{\nlp{\varphi}{2}{\O_n}^2}=0.
\end{equation}

Let us now treat two cases independently:
\subsubsection*{Case 1: $ \varphi \in L^2(\O)$}

In this situation, again by using  Cauchy-Schwarz's inequality,  Jensen's inequality  and Fubini's Theorem,  the inequality \eqref{bcv-eq-liminf-le-lpprimeN1} yields 
$$
I_n\leq \nlp{\varphi}{2}{\O_n} \left[\int_{\O\setminus \O_n}\left(\int_{\O_n}K^2(x,y)\,dx\right)\varphi^2(y)\,dy \right]^{\frac{1}{2}}.$$
Recall that  $K$ satisfies $K(x,y)\le C(1+|x-y|)^{-\alpha}$ for some $C>0$ and $\alpha>3/2$, therefore $p(y):=\int_{\O}K(x,y)\,dx$ is bounded and from the latter inequality we enforce 
$$
I_n\leq C\nlp{\varphi}{2}{\O\setminus \O_n}\nlp{\varphi}{2}{\O_n}.
$$
Thus, 
$$\liminf_{n\to\infty}\frac{I_n}{\nlp{\varphi}{2}{\O_n}^2}\le \liminf_{n\to\infty}\frac{\nlp{\varphi}{2}{\O\setminus \O_n}}{\nlp{\varphi}{2}{\O_n}}=0.$$

\subsubsection*{Case 2: $ \varphi \not\in L^2(\O)$}
Assume now that $\varphi\not \in L^2(\O)$, then we argue as follows.
  Again, applying Fubini's Theorem and Cauchy-Schwarz's inequality in the inequality \eqref{bcv-eq-liminf-le-lpprimeN1} yields
\begin{align}
I_n&\leq \nlp{\varphi}{2}{\O_n} \left[\int_{(\O\cap \R^-)\setminus \O_n}\left(\left(\int_{\O_n}K(x,y)^2\,dx\right)^{\frac{1}{2}}\right)\varphi(y)\,dy+\int_{(\O\cap \R^+) \setminus \O_n}\left(\left(\int_{ \O_n}K(x,y)^2\,dx\right)^{\frac{1}{2}}\right)\varphi(y)\,dy\right], \nonumber\\
&\leq \nlp{\varphi}{2}{\O_n} \left[\tilde I_n^- +\tilde I_n^+\right].\label{bcv-eq-In+}
\end{align}

Recall that by assumption there exists $C>0$ such that $K(x,y)\leq C(1+|x-y|)^{-\alpha}$ with $\alpha>\frac{3}{2}$. So,
we have 
\begin{align*}
&\tilde I_n^-\le C\int_{(\O\cap \R^-)\setminus \O_n}\left(\left(\int_{\O_n}(1+|x-y|)^{-2\alpha}\,dx\right)^{\frac{1}{2}}\right)\varphi(y)\,dy,\\
& \tilde I_n^+\le C\int_{(\O\cap \R^+)\setminus \O_n}\left(\left(\int_{-n}^n(1+|x-y|)^{-2\alpha}\,dx\right)^{\frac{1}{2}}\right)\varphi(y)\,dy.
\end{align*}
To complete our proof, we have to show that  $\frac{\tilde I_n^{\pm}}{\nlp{\varphi}{2}{\O_n}}\to 0$. The proof being similar in both cases, so we only prove that $\frac{\tilde I_n^{+}}{\nlp{\varphi}{2}{\O_n}}\to 0$.
We claim that 

\begin{claim}
 There exists $C>0$ so that for all $n\in\N$, 
 $$\int_{(\O\cap \R^+)\setminus \O_n}\left(\left(\int_{\O_n}(1+|x-y|)^{-2\alpha}\,dx\right)^{\frac{1}{2}}\right)\varphi(y)\,dy\le C.$$ 
\end{claim}
Assume for the moment that the claim holds true, then from   \eqref{bcv-eq-In+},  we deduce that  
$$\frac{I_n}{\nlp{\varphi}{2}{\O_n}^2}\le \frac{C}{\nlp{\varphi}{2}{\O_n}}\to 0 \qquad \text{ when }\quad n\to \infty.$$

Hence, in both situation,  we get
\begin{equation*}
\liminf_{n\to \infty}\lambda_v(\lb{\O_n}+a)\le \mu +\liminf_{n\to\infty}\frac{I_n}{\nlp{\varphi_n}{2}{\O_n}^2}=\lambda_p'(\lb{\O}+a)+\delta 
\end{equation*}
Since $\delta >0$  can be chosen arbitrary,  the above inequality  is true for any $\delta>0$ and the Lemma is proved.

\end{proof}

\begin{proof}[Proof of the Claim]
Since $\varphi\in L^{\infty}(\O)$ and  $y\ge n$ then  $x\le y$ and we have 
\begin{align*}
\tilde I_n^+&\le \|\varphi\|_{\infty}\int_{\O\cap \R^+ \setminus \O_n}\left(\left(\int_{\O_n}(1+y-x)^{-2\alpha}\,dx\right)^{\frac{1}{2}}\right)\,dy,\\
&\le \|\varphi\|_{\infty}\int_n^{+\infty}\left(\left(\int_{-n}^n(1+y-x)^{-2\alpha}\,dx\right)^{\frac{1}{2}}\right)\,dy,\\
&\le \frac{\|\varphi\|_{\infty}}{\sqrt{2\alpha -1}} \int_n^{+\infty}(1+y-n)^{-\alpha+\frac{1}{2}}\,dy,\\
&\le C\int_0^{+\infty}(1+z)^{-\alpha+\frac{1}{2}}\,dz.
\end{align*}

\end{proof}


\section{Asymptotic behaviour of the principal eigenvalue under scaling}\label{bcv-section-scal}

In this section, we investigate further the properties of the principal eigenvalue $\lambda_p(\lb{\O}+a)$ and in particular its behaviour with respect to some scaling of the kernel $K$ ((Proposition \ref{bcv-pev-prop-scal-eq}) and Theorem \ref{bcv-pev-thm4}). For simplicity,  we split this section into two subsections, one  dedicated to the the proof of Proposition \ref{bcv-pev-prop-scal-eq} and the other one  dealing with the proof of Theorem \ref{bcv-pev-thm4}.  
Let us start with the scaling invariance of $\lb{\O}+a$, (Proposition \ref{bcv-pev-prop-scal-eq})
\subsection{Scaling invariance}

This invariance is a consequence of the following observation. By definition  of  $\lambda_p(\lb{\O}+a)$, we have for all $\lambda<\lambda_p(\lb{\O}+a)$, 
$$\oplb{\varphi}{\O}(x)+(a(x)+\lambda)\varphi(x) \le 0 \quad \text{ in }\quad \O,$$
for some positive $\varphi \in C(\O)$.
Let  $ X=\sigma x$, $\O_\sigma:=\frac{1}{\sigma}\O$ and $\psi(X):=\varphi(\sigma X)$ then we can rewrite the above inequality as follows 
\begin{align*}
&\int_{\O}K\left(\frac{X}{\sigma},y\right)\varphi(y)\,dy +(a\left(\frac{X}{\sigma}\right)+\lambda)\varphi\left(\frac{X}{\sigma}\right)\le 0 \quad \text{ for any  } \quad X\in \O_\sigma,\\
&\int_{\O}K\left(\frac{X}{\sigma},y\right)\varphi(y)\,dy +(a_\sigma(X)+\lambda)\psi(X)\le 0 \quad \text{ for any  } \quad X\in \O_\sigma,\\
&\int_{\O_\sigma}K_\sigma\left(X,Y\right)\psi(Y)\,dY +(a_\sigma(X)+\lambda)\psi(X)\le 0 \quad \text{ for any  } \quad X\in \O_\sigma,\\
\end{align*}
with $K_\sigma(x,y):=\frac{1}{\sigma^N} K(\frac{x}{\sigma},\frac{y}{\sigma})$ and $a_\sigma(x):=a\left(\frac{x}{\sigma}\right)$.
Thus $\psi$ is a positive continuous function that satisfies
$$\oplb{\psi}{\sigma,\O_\sigma}(x)+(a_\sigma(x)+\lambda)\psi(x)\le 0 \quad \text{ in }\quad \O_\sigma. $$
Therefore, $\lambda \le \lambda_p(\lb{\sigma,\O_\sigma}+a_\sigma)$ and as a consequence 
$$ \lambda_p(\lb{\O}+a)\le\lambda_p(\lb{\sigma,\O_\sigma}+a_\sigma).$$
Interchanging the role of $\lambda_p(\lb{\O}+a)$ and $\lambda_p(\lb{\sigma,\O_\sigma}+a_\sigma)$ in the above argument yields
$$ \lambda_p(\lb{\O}+a)\ge\lambda_p(\lb{\sigma,\O_\sigma}+a_\sigma).$$

Hence, we get
$$
\lambda_p(\lb{\O}+a)=\lambda_p(\lb{\sigma,\O_\sigma}+a_\sigma).
$$ 

\fdem

\subsection{Asymptotic limits of $\ds{\lambda_p\left(\L_{\sigma,m,\O}-\frac{1}{\sigma^m}+a\right)}$}
Let us  focus on the behaviour of the principal eigenvalue of  the spectral problem
$$\opmem{\varphi}{m,\O }+(a+\lambda)\varphi=0 \quad \text{ in }\quad \O,$$
 where $$\opmem{\varphi}{m,\O}:=\frac{1}{\sigma^m} \left(\int_{\O}J_\sigma(x-y)\varphi(y)\,dy -\varphi(x)\right),$$ with 
 $J_\sigma(z):=\frac{1}{\sigma^N}J\left(\frac{z}{\sigma}\right)$.  
 Assuming that $0\le m\le 2$, we obtain here the limits of $\lambda_p(\M_{\sigma,m}+a)$ when $\sigma \to 0$ and $\sigma\to \infty$. 
 But before going to the study of these limits, we recall a known inequality.
\begin{lemma}\label{bcv-lem-I-le-J}
Let $J\in C(\R^N)$, $J\ge 0$, $J$ symmetric with unit mass, such that $|z|^2J(z) \in L^1(\R^N)$ . Then for all  $\varphi \in H_0^1(\O)$ we have
$$-\int_{\O}\left(\int_{\O}J(x-y)\varphi(y)\,dy-\varphi(x)\right)\varphi(x)\,dx\le \frac{1}{2}\int_{\R^N}J(z)|z|^2\,dz \nlp{\nabla \varphi}{2}{\O}^2.$$ 
\end{lemma}
\begin{proof}

Let $\varphi\in C^{\infty}_c$, then by applying the standard Taylor expansion we have 
\begin{align}
\varphi(x+z)-\varphi(x)&=\int_{0}^1z_i\partial_i \varphi(x+tz)\,dt \label{bcv-eq-taylor1} \\
&=z_i\partial_i\varphi(x)+\int_{0}^1t\left(\int_{0}^1 z_iz_j\partial_{ij} \varphi(x+tsz)\,ds\right)dt\label{bcv-eq-taylor2}
\end{align}
where use the Einstein summation convention $a_ib_i=\sum_{i=1}^Na_ib_i$.\\
Let us denote
$$\I(\varphi):=-\int_{\O}\left(\int_{\O}J(x-y)\varphi(y)\,dy-\varphi(x)\right)\varphi(x)\,dx.$$
 Then, for any $\varphi \in C_c(\O)$, $\varphi \in C_c(\R^N)$ and we can easily see that  
$$\I(\varphi)=\frac{1}{2}\iint_{\R^{2N}}J(x-y)(\varphi(x)-\varphi(y))^2\,dxdy.$$

By plugging the Taylor expansion of $\varphi$ \eqref{bcv-eq-taylor1} in the above equality we see that

\begin{align*}
\frac{1}{2}\iint_{\R^{2N}}J(z)(\varphi(x+z)-\varphi(x))^2\,dzdx&=\frac{1}{2}\int_{\R^N}\int_{\R^N}J(z)\left(\int_0^1 z_i\partial_i\varphi(x+tz)dt\right)^2\,dzdx,\\
&\le\frac{1}{2}\iint_{\R^{2N}}J(z)\left(|z_i|\left[\int_0^1|\partial_i\varphi(x+tz)|^2dt\right]^{\frac{1}{2}}\right)^2\,dzdx,\\
&\le\frac{1}{2}\iint_{\R^{2N}}J(z)\left(\sum_{i=1}^Nz_i^2\right)\left[\sum_{i=1}^N\int_0^1|\partial_i\varphi(x+tz)|^2dt\right]\,dzdx
 \end{align*}
where we use in the last inequality the standard inequality  $(\sum_{i}a_ib_i)^2\le (\sum_{i=1}^Na_i^2)(\sum_{i=1}^Nb_i^2)$.

So, by  Fubini's Theorem and by rearranging the terms in the above inequality, it follows that 
$$
\I(\varphi)\le \frac{1}{2}\left(\int_{\R^N}J(z)|z|^2\,dz \right) \nlp{\nabla \varphi}{2}{\O}^2.
$$
 By density of $C_c^{\infty}(\O)$ in $H^1_0(\O)$, the above  inequality holds true for $\varphi \in H^1_0(\O)$, since obviously the functional $\I(\varphi)$  is continuous in $L^2(\O)$.

\end{proof}

Let us also introduce  the following notation 
\begin{align*}
&J_\sigma(z):=\frac{1}{\sigma^N}J\left(\frac{z}{\sigma}\right), \qquad \quad p_\sigma(x):=\int_{\O}J_\sigma(x-y)\,dy, \qquad \quad  D_2(J):=\int_{\R^N}J(z)|z|^2\,dz,\\
&\a(\varphi):=\frac{\int_{\O}a\varphi^2(x)\,dx} {\nlp{\varphi}{2}{\O}^2},\qquad \quad \r_{\sigma,m}(\varphi):=\frac{1}{\sigma^m}\frac{\int_{\O}(p_\sigma(x)-1)\varphi^2(x)\,dx} {\nlp{\varphi}{2}{\O}^2},\\
&\I_{\sigma,m}(\varphi):=\frac{\frac{1}{\sigma^m}\left(-\int_{\O}\left(\int_{\O}J(x-y)\varphi(y)\,dy-\varphi(x)\right)\varphi(x)\,dx\right)}{\nlp{\varphi}{2}{\O}^2} -\a(\varphi)\\
&\j(\varphi):=\frac{D_2(J)}{2}\frac{\int_{\O}|\nabla \varphi|^2(x)\,dx} {\nlp{\varphi}{2}{\O}^2}.
\end{align*}

With this notation, we see that  
$$\lambda_v(\M_{\sigma,m,\O} +a)=\inf_{\varphi \in L^2(\O)} \I_{\sigma,m}(\varphi), $$

and by Lemma \ref{bcv-lem-I-le-J}, for any $\varphi \in H_0^1(\O)$ we get 

\begin{equation}\label{bcv-eq-I-le-J-R} 
  \I_{\sigma,m}(\varphi)\le \sigma^{2-m}\j(\varphi) -\a(\varphi).
  \end{equation}

We are now in position to obtain the different limits of $\lambda_p(\M_{\sigma,m,\O} +a)$ as $\sigma \to 0$ and $\sigma\to \infty$. For simplicity,  we analyse three distinct situations: $m=0, 0<m<2$ and $m=2$. We will see that $m=0$ and $m=2$ are ,indeed, two critical situations.

 Let us first deal with the easiest case, that is, when $0<m<2$.
 \subsubsection{The case $0<m<2$:}
In this situation, we claim that 

\begin{claim}
Let $\O$ be any  domain and let $J\in C(\R^N)$ be positive, symmetric and such that $|z|^2J(z)\in L^1(\R^N)$. Assume further that $J$ satisfies \eqref{hyp1}--\eqref{hyp3} and $0<m<2$ then 

\begin{align*}
& \lim_{\sigma \to 0} \lambda_p(\M_{\sigma,m,\O} +a)=-\sup_{\O} a\\
& \lim_{\sigma \to +\infty}\lambda_p(\M_{\sigma,m,\O} +a)=-\sup_{\O} a
\end{align*} 

\end{claim}

 \begin{proof}

  First, let us look at the limit  of $\lambda_p$ when $\sigma \to 0$.  
 Up to adding a large positive constant to the function $a$, without any loss of generality,  we can assume that the function $a$ is positive somewhere in $\O$. 
 
 Since $\M_{\sigma,m,\O}+a$ is a self-adjoined operator, by Theorem \ref{bcv-pev-thm2} and \eqref{bcv-eq-I-le-J-R}, for any $\varphi \in H^1_0(\O)$ we have
 $$\lambda_p(\M_{\sigma,m,\O}+a)=\lambda_v(\M_{\sigma,m,\O}+a)\le \I_{\sigma,m}(\varphi)\le \sigma^{2-m}\j(\varphi) -\a(\varphi).$$
 
 Define $\nu:=\sup_{\O} a$, and let $(x_n)_{n\in \N}$  be a sequence of point such that $|\nu -a(x_n)|<\frac{1}{n}$. Since $a$ is positive somewhere, we can also assume that for all $n$, $x_n \in \Gamma:=\{x\in\O \,|\, a(x)>0\}$.  
 
By construction, for any $n>0$, there exists $\rho_n$ such that $B_{\rho}(x_n)\subset \Gamma$ for any positive  $\rho\le \rho_n$. Fix now $n$,  for any $0<\rho\le \rho_n$ there exists $\varphi_\rho \in H_0^1(\O)$  such that $supp(\varphi_\rho)\subset B_\rho(x_n)$ and  therefore, 

 $$\limsup_{\sigma\to 0} \lambda_p(\M_{\sigma,m,\O}+a)\le -\a(\varphi_\rho)=-\frac{\int_{B_\rho(x_n)}a\varphi_\rho^2(x)\,dx}{\nlp{\varphi_\rho}{2}{\O}^2}\le -\min_{B_{\rho}(x_n)}a^+(x).$$  
By taking the limit $\rho \to 0$ in the above inequality, we then get 

 $$\limsup_{\sigma\to 0} \lambda_p(\M_{\sigma,m,\O}+a)\le -a(x_n).$$
Thus, 
$$\limsup_{\sigma\to 0} \lambda_p(\M_{\sigma,m,\O}+a)\le -\nu +\frac{1}{n}.$$
 By sending now $n \to \infty$ in the above inequality, we obtain
 
  $$\limsup_{\sigma\to 0} \lambda_p(\M_{\sigma,m,\O}+a)\le -\nu.$$
  
  On the other hand, by using the test function $(\varphi,\lambda)=(1,-\nu)$ we can easily check that for any $\sigma >0$
  
$$\lambda_p(\M_{\sigma,m,\O}+a)\ge -\nu.$$

Hence, 
 $$-\nu\le \liminf_{\sigma\to 0} \lambda_p(\M_{\sigma,m,\O}+a)\le \limsup_{\sigma\to 0} \lambda_p(\m_{\sigma,m,\O}+a)\le -\nu.$$

Now, let us  look at the limit of $\lambda_p(\m_{\sigma,m,\O}+a)$ when $\sigma \to +\infty$. 
This limit is a straightforward consequence of (iv) of the Proposition \ref{bcv-prop-pev}.
Indeed, as remarked above, for any $\sigma$ by using the test function $(\varphi,\lambda)=(1,-\nu)$, we have
$$-\nu \le \lambda_p(\m_{\sigma,m,\O}+a)$$
whereas from (iv) of the Proposition \ref{bcv-prop-pev} we have 
$$ \lambda_p(\m_{\sigma,m,\O}+a)\le -\sup_{\O}\left(-\frac{1}{\sigma^m}+a\right). $$
Therefore, since $m>0$ we have 

$$-\nu \le \lim_{\sigma\to +\infty} \lambda_p(\m_{\sigma,m,\O}+a)\le -\nu. $$
\end{proof}

\begin{remark}\label{bcv-pev-rem-asymp}.
From the proof, we  obtain also some of the limits in the  cases  $m=0$ and $m=2$.  Indeed, the analysis of the limit of $\lambda_p(\m_{\sigma,m,\O}+a)$ when $\sigma \to 0$ holds true as soon as $m<2$. Thus,  $$\lambda_p(\m_{\sigma,0,\O}+a)\to -\sup_{\O} a \quad\text{ as } \quad\sigma \to 0.$$
On the other hand, the  analysis of the  limit of $\lambda_p(\m_{\sigma,m,\O}+a)$ when $\sigma \to +\infty$ holds true as soon as $m>0$. Therefore, $$\lambda_p(\m_{\sigma,2,\O}+a)\to -\sup_{\O} a \quad\text{ as } \quad\sigma \to +\infty.$$ 

\end{remark}
\subsubsection{The case $m=0$}
In this situation,   one of the above argument fails and  one of the expected limits  is not $-\nu$ any more. Indeed, we have

\begin{lemma}
Let $\O$ be any  domain and let $J\in C(\R^N)$ be positive, symmetric and such that $|z|^2J(z)\in L^1(\R^N)$. Assume further that $J$ satisfies \eqref{hyp1}--\eqref{hyp3} and $m=0$ then 
\begin{align*}
& \lim_{\sigma \to 0} \lambda_p(\m_{\sigma,0,\O} +a)=-\sup_{\O} a\\
& \lim_{\sigma \to +\infty}\lambda_p(\m_{\sigma,0,\O} +a)=1-\sup_{\O} a
\end{align*} 

\end{lemma}
 
\begin{proof}
As already noticed in Remark \ref{bcv-pev-rem-asymp}, the limit of $\lambda_p(\m_{\sigma,0,\O}+a)$ when $\sigma \to 0$ can be obtained by following the arguments developed in the case $0<m<2$. 
 
Therefore, it remains only to establish the limit of $\lambda_p(\m_{\sigma,0,\O}+a)$ when $\sigma \to \infty$.

As above, up to adding a large positive constant to $a$, without any loss of generality, we can assume that $a$ is positive somewhere in $\O$ and we denote  $\nu:=\sup_{\O}a>0$. 
By using constant test functions and (iv) of the Proposition \ref{bcv-prop-pev},    we observe that

$$-\nu\le \lambda_p(\m_{\sigma,0,\O}+a)\le 1 -\nu, \quad \text{ for all }\quad \sigma >0.$$
So, we have $$\limsup_{\sigma\to \infty}\lambda_p(\m_{\sigma,0,\O}+a)\le 1 -\nu.$$
On the other hand, for any $\varphi \in C_c(\O)$  we have for all $\sigma$,

\begin{align*}
\I_{\sigma,0}(\varphi) \int_{\O}\varphi^2(x)&=-\int_{\O}\left(\int_{\O} J_\sigma(x-y)\varphi(y)\,dy-\varphi(x)\right)\varphi(x)dx -\int_{\O}a\varphi^2(x)\,dx,\\
&=-\iint_{\O\times \O} J_\sigma(x-y)\varphi(x)\varphi(y)dxdy +\int_{\O}\varphi^2(x)\,dx -\int_{\O}a\varphi^2(x)\,dx,\\
&\ge -\nlp{\varphi}{2}{\O}\left(\int_{\O}\left( \int_{\O} J_\sigma(x-y)\varphi(x)\,dx\right)^2\,dy\right)^{1/2} +\int_{\O}\varphi^2(x)\,dx -\sup_{\O}a\int_{\O}\varphi^2(x)\,dx,\\
&\ge -\sqrt{\|J_\sigma\|_{\infty}}\nlp{\varphi}{2}{\O}^2+\int_{\O}\varphi^2(x)\,dx -\nu\int_{\O}\varphi^2(x)\,dx,\\
&\ge \left(-\frac{\sqrt{\|J\|_{\infty}}}{\sigma^{N/2}}+1-\nu \right) \int_{\O}\varphi^2(x)\,dx.
\end{align*}

Thus, for all $\sigma$ we have 
$$ \I_{\sigma,0}(\varphi)\ge \left(-\frac{\sqrt{\|J\|_{\infty}}}{\sigma^{N/2}}+1-\nu \right).$$
By density of $C_c(\O)$ in  $L^2(\O)$, the above inequality holds  for any $\varphi \in L^2(\O)$.

Therefore, by Theorem \ref{bcv-pev-thm2} for all $\sigma$
$$\lambda_p(\m_{\sigma,0,\O} +a)=\lambda_v(\m_{\sigma,0,\O} +a)\ge -\frac{\sqrt{\|J\|_{\infty}}}{\sigma^{N/2}}+ 1 -\nu,$$
 and 
 $$ \liminf_{\sigma\to +\infty} \lambda_p(\m_\sigma +a)\ge 1 -\nu.$$
 Hence,
 $$1 -\nu\le \liminf_{\sigma\to +\infty} \lambda_p(\m_{\sigma,0,\O} +a)\le \limsup_{\sigma\to +\infty} \lambda_p(\m_{\sigma,0,\O} +a) \le 1 -\nu.$$
\end{proof}

To conclude this subsection, we analyse the  monotonic behaviour of $\lambda_p(\m_{\sigma,0,\O}+a)$ with respect to $\sigma$ in the particular case  $\O=\R^N$. More precisely,

\begin{proposition}
Let $\O=\R^N,$ $a\in C(\R^N)$ and  $J\in C(\R^N)$ be positive, symmetric and such that $|z|^2J(z)\in L^1(\R^N)$. Assume further that $J$ satisfies \eqref{hyp1}--\eqref{hyp3}, $m=0$ and $a$ is  symmetric ($a(x)=a(-x)$ for all $x$) and the map $t\to a(tx)$ is non increasing for all $x,t>0$. Then the map $\sigma \to \lambda_p(\sigma)$ is monotone non decreasing. 
 \end{proposition}
\begin{proof}
When $\O=\R^N$, thanks to Proposition \ref{bcv-pev-prop-scal-eq}, we have 
$$\lambda_p(\m_{\sigma,0,\R^N}+a)=\lambda_p(\m_{1,0,\R^N} +a_\sigma(x)). $$
Since the function $a_\sigma(x)$ is monotone non increasing with respect to $\sigma$, by (i) of Proposition \ref{bcv-prop-pev}, for all $\sigma \ge \sigma^*$ we have 
 $$\lambda_p(\m_{\sigma^*,0,\R^N}+a) =\lambda_p(\m_{1,0,\R^N}+a_{\sigma^*}(x))\le\lambda_p(\M_{1,0,\R^N}+a_{\sigma}(x))=\lambda_p(\m_{\sigma,0,\R^N}+a).$$
\end{proof}

\subsubsection{The case $m=2$}
 Finally, let us study the case $m=2$ and end the proof of Theorem \ref{bcv-pev-thm4}. In this situation, we claim that 
 
 \begin{lemma}
Let $\O$ be a domain, $a\in C(\O)$ and let $J\in C(\R^N)$ be positive, symmetric and such that $|z|^2J(z)\in L^1(\R^N)$. Assume further that $J$ satisfies \eqref{hyp1}--\eqref{hyp3} , $a \in C^{0,\alpha}(\O)$ with $\alpha>0$ and $m=2$ then 

\begin{align}
& \lim_{\sigma \to +\infty}\lambda_p(\m_{\sigma,2,\O} +a)=-\sup_{\O} a, \nonumber\\
& \lim_{\sigma \to 0} \lambda_p(\m_{\sigma,2,\O} +a)=\lambda_1\left(\frac{D_2(J)K_{2,N}}{2}\Delta +a,\O\right), \label{bcv-pev-lim-m-2-0}
\end{align} 
where $$K_{2,N}:=\frac{1}{|S^{N-1}|}\int_{S^{N-1}}(s.e_1)^2\,ds= \frac{1}{N}$$

and $$\lambda_1\left(K_{2,N}D_2(J)\Delta +a,\O\right):=\inf_{\varphi \in H^1_0(\O),\varphi\not \equiv 0} K_{2,N}\j(\varphi) -\a(\varphi). $$
\end{lemma}

\begin{proof}
 In this situation, as already noticed in Remark \ref{bcv-pev-rem-asymp}, by following the arguments used  in the case $2>m>0$,  we can obtain the limit of $\lambda_p(\m_{\sigma,2,\O} +a)$ as $\sigma \to \infty$.  So, it  remains to prove \eqref{bcv-pev-lim-m-2-0}.

Let us rewrite $\I_{\sigma,2}(\varphi)$ in a more convenient way. Let $\rho_\sigma(z):=\frac{1}{\sigma^2D_{2}(J)}J_\sigma(z)|z|^2$, then for $\varphi\in H^1_0(\O)$, we have

\begin{align}
\I_{\sigma,2}(\varphi)&=\frac{1}{\nlp{\varphi}{2}{\O}^2}\left(\frac{1}{2\sigma^2}\iint_{\O\times \O}J_\sigma(x-y)(\varphi(x)-\varphi(y))^2\,dxdy\right)-\r_\sigma(\varphi)-\a(\varphi),\\
&=\frac{1}{\nlp{\varphi}{2}{\O}^2}\left(\frac{D_2(J)}{2}\iint_{\O\times \O}\rho_\sigma(x-y)\frac{(\varphi(x)-\varphi(y))^2}{|x-y|^2}\,dxdy\right)-\r_\sigma(\varphi) -\a(\varphi).
\end{align}

We are now is position to prove \eqref{bcv-pev-lim-m-2-0}. Let us first show that 
\begin{equation}\label{bcv-pev-lim-m-2-0-+}
\limsup_{\sigma\to 0}\lambda_p(\m_{\sigma,2,\O}+a)\le \lambda_1\left(\frac{K_{2,N}D_2(J)}{2}\Delta +a,\O\right).
\end{equation}

This inequality  follows from the two following observations.

First, for any $\o  \subset \O$  compact subset of $\O$,  we have for $\sigma$ small enough 
$$p_\sigma(x)=\int_{\O}J_\sigma(x-y)\,dy=1 \quad \text{ for all }\quad x \in \o.$$
Therefore,  for $\varphi \in C_c^{\infty}(\O)$ and $\sigma$ small enough,   
\begin{equation} \label{bcv-pev-eq-asymp-m-2-2}
\r_{\sigma,2}(\varphi)=\frac{1}{\sigma^{2}\nlp{\varphi}{2}{\O}^2}\int_{\O} (p_\sigma(x)-1)\varphi^2(x) \,dx=0.
\end{equation}
Secondly, by definition, $\rho_\sigma$ is a continuous mollifier such that
$$\begin{cases}
\rho_\sigma\ge 0\quad \text{ in }\quad \R^N,\\
\int_{\R^N}\rho_\sigma(z)dz=1, \quad \forall \, \sigma >0,\\
\lim_{\sigma\to 0}\int_{|z|\ge \delta}\rho_\sigma(z)dz=0, \quad \forall \, \delta >0, 
\end{cases}
$$
which, from the characterisation of Sobolev spaces in \cite{Bourgain2001,Brezis2002,Ponce2004},  enforces that 
\begin{equation} \label{bcv-pev-eq-asymp-m-2-1}
\lim_{\sigma\to 0} \iint_{\O\times \O}\rho_\sigma(x-y)\frac{(\varphi(x)-\varphi(y))^2}{|x-y|^2}\,dxdy =K_{2,N}\nlp{\nabla\varphi}{2}{\O}^2, \quad \text{for any } \quad  \varphi \in H^1_0(\O).
\end{equation}

Thus, for any $\varphi \in C^{\infty}_c(\O)$
$$\limsup_{\sigma\to 0}\lambda_p(\m_{\sigma,2,\O}+a)\le\lim_{\sigma\to 0}\I_{\sigma,2}(\varphi)  =K_{2,N}\j(\varphi)-\a(\varphi).$$

From the above inequality, by  definition of $\lambda_1\left(\frac{K_{2,N}D_2(J)}{2}\Delta +a,\O\right)$, it is then standard to obtain  
$$\limsup_{\sigma\to 0}\lambda_p(\m_{\sigma,2,\O}+a)\le \lambda_1\left(\frac{K_{2,N}D_2(J)}{2}\Delta +a,\O\right).$$

To complete our proof, it remains to establish the following inequality

$$\lambda_1\left(\frac{K_{2,N}D_2(J)}{2}\Delta +a,\O\right)\le \liminf_{\sigma\to 0}\lambda_p(\m_{\sigma,2,\O}+a).$$

Observe that to obtain  the above inequality, it is sufficient to prove that 

 \begin{equation}\label{bcv-pev-eq-m2-eps0-liminf}
 \lambda_1\left(\frac{K_{2,N}D_2(J)}{2}\Delta +a,\O\right)\le \liminf_{\sigma\to 0}\lambda_p(\m_{\sigma,2,\O}+a)+2\delta \quad \text{for all }\quad \delta>0.
 \end{equation}
  
Let us fix $\delta>0$. Now, to obtain \eqref{bcv-pev-eq-m2-eps0-liminf}, we construct adequate smooth test functions $\varphi_\sigma$ and estimate $K_{2,N}\j(\varphi_\sigma)-\a(\varphi_\sigma)$ in terms of $\lambda_p(\m_{\sigma,2,\O}+a),\delta$ and some reminder $R(\sigma)$ that converges to $0$ as $\sigma \to 0$. Since our argument is rather long, we decompose it into three steps.
\subsection*{Step One: Construction of a good the test function}

We first claim that, for all $\sigma>0$, there exists $\varphi_\sigma\in C_c^{\infty}(\O)$ such that 
$$\opmem{\varphi_\sigma}{2,\O}(x) +(a(x)+\lambda_p(\m_{\sigma,2,\O}+a)+2\delta)\varphi_\sigma(x)\ge 0 \quad \text{ for all }\quad x \in \O.$$
 Indeed, by Theorem \ref{bcv-pev-thm2}, we have $\lambda_p(\m_{\sigma,2,\O}+a)=\lambda_p''(\m_{\sigma,2,\O}+a)$, therefore  for all $\sigma$, there exists $\psi_\sigma\in C_c(\O)$ such that 
$$\opmem{\psi_\sigma}{2,\O}(x) +(a(x)+\lambda_p(\m_{\sigma,2,\O}+a)+\delta)\psi_\sigma(x)\ge 0 \quad \text{ for all }\quad x \in \O.$$
Since $\psi_\sigma \in C_c(\O)$, we can easily check that 
$$\opmem{\psi_\sigma}{2,\R^N}(x) +(a(x)+\lambda_p(\m_{\sigma,2,\O}+a)+\delta)\psi_\sigma(x)\ge 0 \quad \text{ for all }\quad x \in \R^N.$$
Now, let $\eta$ be a smooth  mollifier of unit mass and with support in the unit ball and consider  $\eta_\tau:=\frac{1}{\tau^N}\eta\left(\frac{z}{\tau}\right)$ for $\tau>0$. 

 By taking $\tilde\varphi_\sigma:= \eta_\tau\star \psi_\sigma$ and  observing that $\opmem{\tilde\varphi_\sigma}{2,\R^N}(x)=\eta_{\tau}\star(\opmem{\psi_\sigma}{2,\R^N})(x) $ for any $x\in \R^N$, we deduce that 

   \begin{align*}
   &\eta_\tau \star\left(\opmem{\psi_\sigma}{2,\R^N} +(a(x)+\lambda_p(\m_{\sigma,2,\O}+a)+\delta)\psi_\sigma\right)\ge 0 \quad \text{ for all }\quad x \in \R^N, \\
   & \opmem{\tilde\varphi_\sigma}{2,\R^N}(x) +(\lambda_p(\m_{\sigma,2,\O}+a)+\delta)\tilde\varphi_\sigma(x) +\eta_\tau\star (a\psi_\sigma)(x)\ge 0 \quad \text{ for all }\quad x \in \R^N.
   \end{align*}
  By adding and subtracting $a$,  we then have,  for all $x\in \R^N$,
  $$
    \opmem{\tilde\varphi_\sigma}{2,\R^N}(x) +(a(x)+\lambda_p(\m_{\sigma,2,\O}+a)+\delta)\tilde\varphi_\sigma(x) + \int_{\R^N}\eta_\tau(x-y)\psi_\sigma(y)(a(y)-a(x))\,dy\ge 0.
   $$
  For $\tau$ small enough, say $\tau \le \tau_0$, the function   $\tilde\varphi_\sigma\in C^{\infty}_{c}(\O)$ and for all $x\in \O$ we have 
 \begin{align*}
 \opmem{\tilde\varphi_\sigma}{2,\R^N}(x)&=\frac{1}{\sigma^2}\left(\int_{\R^N}J_\sigma(x-y)\tilde\varphi_\sigma(y)\,dy-\tilde\varphi_\sigma(x)\right),\\
 &=\frac{1}{\sigma^2}\left(\int_{\O}J_\sigma(x-y)\tilde\varphi_\sigma(y)\,dy-\tilde\varphi_\sigma(x)\right)= \opmem{\tilde\varphi_\sigma}{2,\O}(x).
 \end{align*}  
   Thus, from the above inequalities, for $\tau \le \tau_0$, we get for all $ x \in \O,$
   $$  \opmem{\tilde\varphi_\sigma}{2,\O}(x) +(a(x)+\lambda_p(\m_{\sigma,2,\O}+a)+\delta)\tilde\varphi_\sigma(x) + \int_{\R^N}\eta_\tau(x-y)\psi_\sigma(y)(a(y)-a(x))\,dy\ge 0.$$
   Since $a$ is H\"older continuous, we can estimate the integral  by  
   \begin{align*}
   \left|\int_{\R^N}\eta_\tau(x-y)\psi_\sigma(y)(a(y)-a(x))\,dy \right|&\le \int_{\R^N}\eta_\tau(x-y)\psi_\sigma(y)\left|\frac{a(y)-a(x)}{|y-x|^{\alpha}}\right||x-y|^{\alpha}\,dy,\\
&\le \kappa\tau^{\alpha}\tilde\varphi_\sigma(x),
   \end{align*}
   where $\kappa$ is the H\"older semi-norm of $a$.
   Thus, for $\tau$  small, says $\tau\le \inf\{\left(\frac{\delta}{2 \kappa}\right)^{1/\alpha},\tau_0\}$, we have
     \begin{equation}\label{bcv-eq-lp-eps2-1}
   \opmem{\tilde\varphi_\sigma}{2,\O}(x) +(a(x)+\lambda_p(\m_{\sigma,2,\O}+a)+2\delta)\tilde\varphi_\sigma(x)\ge 0 \quad \text{ for all } \quad x\in \O.
   \end{equation}
   
Let us consider now $\varphi_\sigma:=\gamma \tilde\varphi_\sigma$, where $\gamma$ is a positive constant to be chosen.  From \eqref{bcv-eq-lp-eps2-1}, we obviously have 
 \begin{equation}\label{bcv-eq-lp-eps2-1bis}
   \opmem{\varphi_\sigma}{2,\O}(x) +(a(x)+\lambda_p(\m_{\sigma,2,\O}+a)+2\delta)\varphi_\sigma(x)\ge 0 \quad \text{ for all } \quad x\in \O.
   \end{equation}

By taking $\gamma:=\frac{\int_{\R^N}\psi_\sigma^2(x)\,dx}{\int_{\R^N}\tilde\varphi_\sigma^2(x)\,dx}$, we get 
\begin{equation} \label{bcv-eq-lp-eps2-norm}
\frac{\int_{\R^N}\psi_\sigma^2(x)\,dx}{\int_{\R^N}\varphi_\sigma^2(x)\,dx}=1.     
    \end{equation}
   
\subsection*{Step Two: A first estimate of $\lambda_1$}  
 Now, by multiplying $\opmem{\varphi_\sigma}{2,\O}$  by $-\varphi_\sigma$ and integrating  over $\O$,  we then get

\begin{align}
    -\int_{\O}\opmem{\varphi_\sigma}{2,\O}\varphi_\sigma(x)\,dx&= -\iint_{\R^N\times\R^N}\frac{1}{\sigma^2}J_\sigma(x-y)(\varphi_\sigma(y)-\varphi_\sigma(x))\varphi_\sigma(x)\,dydx,\\
    &=\frac{1}{2\sigma^2}\iint_{\R^N\times\R^N}J_\sigma(x-y)(\varphi_\sigma(y)-\varphi_\sigma(x))^2\,dxdy,   \\
    &=\frac{D_2(J)}{2}\iint_{\R^N\times\R^N}\rho_\sigma(z)\frac{(\varphi_\sigma(x+z)-\varphi_\sigma(x))^2}{|z|^2}\,dzdx.\label{bcv-eq-lp-eps2-2}
    \end{align}

    By combining \eqref{bcv-eq-lp-eps2-1} and \eqref{bcv-eq-lp-eps2-2} we therefore obtain  
    \begin{multline}\label{bcv-eq-lp-eps2-3}
   \frac{D_2(J)}{2}\iint_{\R^N\times\R^N}\rho_\sigma(z)\frac{(\varphi_\sigma(x+z)-\varphi_\sigma(x))^2}{|z|^2}\,dzdx-\int_{\R^N}a(x)\varphi_\sigma^2(x)\,dx  \\ \le (\lambda_p(\m_{\sigma,2,\O}+a)+2\delta)\int_{\R^N}\varphi_\sigma^2(x)\,dx.    
    \end{multline}

  On the other hand, inspired by the proof of Theorem 2  in \cite{Brezis2002}, since $\varphi_\sigma \in C^{\infty}_c(\R^N)$,  by Taylor's expansion, for all $x,z \in \R^N$,  we have
$$|\varphi_\sigma(x+z)-\varphi_\sigma(x) -z\cdot\nabla \varphi_\sigma(x)|\le \sum_{i,j}|z_iz_j|\int_0^1 t\left(\int_0^1|\partial_{ij}\varphi_\sigma(x+tsz)|\,ds \right)\,dt.$$

  Therefore, 
  $$ |z\cdot\nabla \varphi_\sigma(x)|\le \sum_{i,j}|z_iz_j|\int_0^1 t\left(\int_0^1|\partial_{ij}\varphi_\sigma(x+tsz)|\,ds \right)\,dt+ |\varphi_\sigma(x+z)-\varphi_\sigma(x)|,$$
  and for every  $\theta> 0$ we have 
  
  \begin{align*} 
  |z\cdot\nabla \varphi_\sigma(x)|^2&\le C_{\theta}\left[\sum_{i,j}|z_iz_j|\int_0^1 t\left(\int_0^1|\partial_{ij}\varphi_\sigma(x+tsz)|\,ds \right)\,dt\right]^2+ (1+\theta)|\varphi_\sigma(x+z)-\varphi_\sigma(x)|^2,\\
  & \le C_{\theta}\sum_{i,j}|z_iz_j|^2\iint_{[0,1]^2} t^2|\partial_{ij}\varphi_\sigma(x+tsz)|^2\,ds dt+ (1+\theta)|\varphi_\sigma(x+z)-\varphi_\sigma(x)|^2.
  \end{align*}

Thus, by integrating in $x$ and $z$ over $\R^{N}\times \R^N$, we get 
\begin{equation*}
\begin{split} 
\iint \frac{\rho_\sigma(|z|)}{|z|^2}|z\cdot\nabla \varphi_\sigma(x)|^2\,dzdx \le  C_{\theta} \iint \rho_\sigma(|z|)\sum_{i,j}\frac{|z_iz_j|^2}{|z|^2}\left(\iint_{[0,1]^2} t^2|\partial_{ij}\varphi_\sigma(x+tsz)|^2\,ds dt\right)\,dzdx 
\\+ (1+\theta)\iint\rho_\sigma(|z|)\frac{|\varphi_\sigma(x+z)-\varphi_\sigma(x)|^2}{|z|^2}\,dzdx. 
\end{split}
\end{equation*}  
For $\sigma$ small,  $supp(\rho_\sigma)\subset B_1(0)$, and   we have for all $x \in \R^N$,
 $$\int_{\R^N}\frac{\rho_\sigma(|z|)}{|z|^2}|z\cdot\nabla \varphi_\sigma(x)|^2\,dz= K_{2,N}|\nabla \varphi_\sigma(x)|^2,$$
 
  whence, 
  \begin{equation}
  \begin{split}
 K_{2,N}\int_{\R^N}|\nabla \varphi_\sigma(x)|^2\, dx \le   C_{\theta} \iint\rho_\sigma(|z|)\sum_{i,j}\frac{|z_iz_j|^2}{|z|^2}\left(\iint_{[0,1]^2} t^2|\partial_{ij}\varphi_\sigma(x+tsz)|^2\,ds dt\right)\,dzdx 
\\+ (1+\theta) \iint\rho_\sigma(|z|)\frac{|\varphi_\sigma(x+z)-\varphi_\sigma(x)|^2}{|z|^2}\,dzdx. \label{bcv-eq-lp-eps2-4}
\end{split}
 \end{equation}

 Dividing \eqref{bcv-eq-lp-eps2-4} by $\nlp{\varphi_\sigma}{2}{\O}^2$ and then subtracting $\a(\varphi_\sigma)$ on both side, we get  

\begin{equation}
 K_{2,N}\j(\varphi_\sigma)-\a(\varphi_\sigma)\le  R(\sigma) +(1+\theta) \I_{\sigma,2}(\varphi_\sigma)+\theta\a(\varphi_\sigma), \label{bcv-eq-lp-eps2-5}
 \end{equation} 

where $R(\sigma)$ is defined by $$ R(\sigma):=\frac{C_{\theta}}{\nlp{\varphi_\sigma}{2}{\O}^2} \iint\rho_\sigma(|z|)\sum_{i,j}\frac{|z_iz_j|^2}{|z|^2}\left(\iint_{[0,1]^2} t^2|\partial_{ij}\varphi_\sigma(x+tsz)|^2\,dsdt\right)\,dzdx. $$ 
 
 By combining now \eqref{bcv-eq-lp-eps2-5} with \eqref{bcv-eq-lp-eps2-3},  by definition of $\lambda_1\left(\frac{K_{2,N}D_2(J)}{2}\Delta +a,\O\right)$, we obtain

\begin{equation}
 \lambda_1\left(\frac{K_{2,N}D_2(J)}{2}\Delta +a,\O\right)\le R(\sigma) + (1+\theta)[\lambda_p(\M_{\sigma,2,\O}+a)+2\delta]+\theta\nli{a}{\O}.   \label{bcv-eq-lp-eps2-6}
 \end{equation}  
 
\subsection*{Step Three: Estimates of $R(\sigma)$ and conclusion}
 
 Let us now estimate $R(\sigma)$ and finish our argument. 
 
 By construction, we have $\partial_{ij}\varphi_\sigma=\partial_{ij}\eta_\tau \star \psi_\sigma$. So, by Fubini's Theorem and standard convolution estimates, we get for $\sigma$ small

\begin{align*}
R(\sigma)
&\le \sum_{i,j} \int_{|z|\le 1}\iint_{[0,1]^2} \rho_\sigma(|z|)\frac{|z_iz_j|^2}{|z|^2} t^2\left(\int_{\R^N}|\partial_{ij}\eta_\tau \star \psi_\sigma(x+tsz)|^2\,dx\right)dtdsdz,\\
&\le \left(\int_{|z|\le 1}\int_{[0,1]} \rho_\sigma(|z|)\sum_{i,j}\frac{|z_iz_j|^2}{|z|^2} t^2\,dtdz\right) \nlp{\nabla^2\eta_\tau}{1}{\R^N}\nlp{\psi_\sigma}{2}{\R^N}^2, \\
&\le \frac{2}{3}\nlp{\nabla^2\eta_\tau}{1}{\R^N}\nlp{\psi_\sigma}{2}{\R^N}^2 \int_{|z|\le 1}\rho_\sigma(|z|)|z|^2\,dz.
 \end{align*}

  Combining this inequality with \eqref{bcv-eq-lp-eps2-6}, we get   
  
  \begin{multline*}
 \lambda_1\left(\frac{K_{2,N}D_2(J)}{2}\Delta +a,\O\right)\le (1+\theta)[\lambda_p(\M_{\sigma,2,\O}+a)+2\delta]+\theta \nli{a}{\O} \\+
   \frac{2C_{\theta}}{3}\nlp{\nabla^2\eta_\tau}{1}{\R^N}\frac{\nlp{\psi_\sigma}{2}{\R^N}^2 }{\nlp{\varphi_\sigma}{2}{\O}^2} \int_{|z|\le 1}\rho_\sigma(|z|)|z|^2dz.   
 \end{multline*} 
  Since $\varphi_\sigma \in C_c^{\infty}(\O)$, $\nlp{\varphi_\sigma}{2}{\O}^2=\nlp{\varphi_\sigma}{2}{\R^N}^2$ and  thanks to  \eqref{bcv-eq-lp-eps2-norm}, the above inequality reduces to  
  
  \begin{multline}
 \lambda_1\left(\frac{K_{2,N}D_2(J)}{2}\Delta +a,\O\right)\le (1+\theta)[\lambda_p(\M_{\sigma,2,\O}+a)+2\delta]+\theta \nli{a}{\O} \\+
   \frac{2C_{\theta}}{3}\nlp{\nabla^2\eta_\tau}{1}{\R^N} \int_{|z|\le 1}\rho_\sigma(|z|)|z|^2dz.    \label{bcv-eq-lp-eps2-7}
 \end{multline} 
  
 Now,   since $\int_{|z|\le 1}\rho_\sigma(|z|)|z|^2dz\le \sigma^2$, letting $\sigma \to 0$ in \eqref{bcv-eq-lp-eps2-7} yields
  
    \begin{equation}
 \lambda_1\left(\frac{K_{2,N}D_2(J)}{2}\Delta +a,\O\right)\le (1+\theta)[2\delta+\liminf_{\sigma \to 0}\lambda_p(\M_{\sigma,2,\O}+a)]+\theta \nli{a}{\O}.   \label{bcv-eq-lp-eps2-8}
 \end{equation}

Since \eqref{bcv-eq-lp-eps2-8} holds for every $\theta$, we obtain

    \begin{equation*}
 \lambda_1\left(\frac{K_{2,N}D_2(J)}{2}\Delta +a,\O\right)\le \liminf_{\sigma \to 0}\lambda_p(\M_{\sigma,2,\O}+a) +2\delta.   \label{bcv-eq-lp-eps2-9}
 \end{equation*}

\end{proof}

\section{Asymptotics of $\varphi_{p,\sigma}$}

In this last section, we investigate  the existence of a positive continuous eigenfunction  $\varphi_{p,\sigma}$ associated to the principal eigenvalue $ \lambda_p(\M_{\sigma,2,\O}+a)$.

The existence of such a $\varphi_{p,\sigma}$ is a straightforward consequence of the existence criteria in bounded domain (Theorem \ref{bcv-pev-th-crit2}) and the asymptotic behaviour of the principal eigenvalue (Theorem \ref{bcv-pev-thm4}).

Indeed, assume first that  $\O$ is bounded, then since $a \in L^{\infty}(\bar \O)$, there exists  $\sigma_0$ such that for all  $\sigma\le \sigma_0$, 
$$ \frac{1}{\sigma^2}-\sup_{\O}a> 1+\left|\lambda_1\left(\frac{K_{2,N}D_2(J)}{2}\Delta +a, \O\right)\right|.$$ 

Now,  thanks to $\lambda_p(\m_{\sigma,2,\O}+a)\to \lambda_1\left(\frac{K_{2,N}D_2(J)}{2}\Delta +a, \O\right)$,  for $\sigma$ small enough, says $\sigma  \le \sigma_1$, we get
$$\lambda_p(\m_{\sigma,2,\O}+a)\le 1+\left|\lambda_1\left(\frac{K_{2,N}D_2(J)}{2}\Delta +a, \O\right)\right|.$$

Thus, for $\sigma \le \inf\{\sigma_1,\sigma_0\}$, 
$$\lambda_p(\m_{\sigma,2,\O}+a)< \frac{1}{\sigma^2}-\sup_{\O}a,$$
which, thanks to Theorem \ref{bcv-pev-th-crit2}, enforces the existence of a principal positive continuous eigenfunction $\varphi_{p,\sigma}$ associated with $ \lambda_p(\m_{\sigma,2,\O}+a)$.

From the above argument, we can easily obtain the existence of eigenfunction when  $\O$ is unbounded. Indeed, let $\O_0$ be a bounded sub-domain of $\O$ and let $\gamma:=\sup\{\left|\lambda_1\left(\O_0\right)\right|;\left|\lambda_1\left(\O\right)\right| \}.$ Since $a$ is bounded in $\O$, there exists $\sigma_0$ such that
for all  $\sigma\le \sigma_0$, 
$$ \frac{1}{\sigma^2}-\sup_{\O}a> 2+\gamma.$$ 
As above, since $\lambda_p(\m_{\sigma,2,\O_0}+a)\to \lambda_1(\O_0)$, there exists $\sigma_1$ such that for all $\sigma\le \sigma_1$ we have 
$$\lambda_p(\m_{\sigma,2,\O_0}+a)\le 1+\gamma.$$ 

For any bounded domain $\O'$ such that $\O_0\subset \O'\subset \O$, by  monotonicity of $ \lambda_p(\m_{\sigma,2,\O'}+a)$  with respect to $\O'$, for all $\sigma \le \sigma_1$ we have 
 $$\lambda_p(\m_{\sigma,2,\O'}+a)\le 1+\gamma.$$ 
Therefore, for all $\sigma \le \sigma_2:=\inf\{\sigma_0,\sigma_1\}$, we have 
 
$$\lambda_p(\m_{\sigma,2,\O'}+a)+1\le  \frac{1}{\sigma^2}-\sup_{\O'}a,$$ 
and thus, thanks to Theorem \ref{bcv-pev-th-crit2}, for all $ \sigma\le \sigma_2$ there exists  $\varphi_{p,\sigma}$ associated to $\lambda_p(\m_{\sigma,2,\O'}+a)$.  

To construct a positive eigenfunction $\varphi_{p,\sigma}$ associated to $\lambda_p(\m_{\sigma,2,\O}+a)$, we then argue as follows.
 
Let $(\O_n)_{n\in \N}$ be an increasing sequence of bounded sub-domain of $\O$ that converges to $\O$. Then, for all $\sigma\le \sigma_2$, for each $n$ there exists a continuous positive function $\varphi_{n,\sigma}$ associated to $\lambda_p(\m_{\sigma,2,\O_n}+a)$. Without any loss of generality, we can assume that $\varphi_n$ is normalised by $\varphi_n(x_0)=1$ for some fixed $x_0\in  \O_0$. Since for all $n, \lambda_p(\m_{\sigma,2,\O_n}+a)+1\le  \frac{1}{\sigma^2}-\sup_{\O_n}a$, the Harnack inequality applies to $\varphi_n$ and thus the sequence $(\varphi_n)_{n\in \N}$ is locally uniformly bounded in $C^0$ topology. By a standard diagonal argument, there exists a subsequence, still denoted $(\varphi)_{n\in \N}$, that converges point-wise to some non-negative function $\varphi$. Thanks to the Harnack inequality, $\varphi$ is positive. Passing to the limit in the equation satisfied by $\varphi_n$, thanks to the Lebesgue dominated convergence Theorem, $\varphi$ satisfies 
$$\opmem{\varphi}{2,\O}(x)+(a(x)+\lambda_{p,\sigma}(\m_{\sigma,2,\O}+a)) \varphi(x)=0 \quad \text{ for all }\quad x \in \O.$$
Since $a$ is continuous and $((a(x)+\lambda_{p,\sigma}(\m_{\sigma,2,\O}+a))-\frac{1}{\sigma^2})<0$, we deduce that $\varphi$ is also continuous. Hence, $\varphi$ is a positive continuous eigenfunction associated with $\lambda_{p,\sigma}(\m_{\sigma,2,\O}+a)$.

\begin{remark}
We observe that such  arguments hold also for the operators $\M_{\sigma,m,\O}+a$ with $0<m<2$, since in such cases, $\lambda_p(\sigma)<+\infty$ for all $\sigma$ and $-\sup_{\O}(-\frac{1}{\sigma^2}+a)\to +\infty$. Thus, when $0<m<2$, for $\sigma(m)$ small enough, there exists always a positive function $\varphi_{p,\sigma}\in C(\bar\O)$ associated with   $ \lambda_p(\m_{\sigma,m,\O}+a)$.
\end{remark}

Finally, let us complete the proof of Theorem \ref{bcv-pev-thm5} by obtaining  the asymptotic behaviour of $\varphi_{p,\sigma}$ when $\sigma \to 0$ assuming that $\varphi_{p,\sigma}\in L^2(\O)$. 
We first recall the following  useful identity :
\begin{proposition}\label{bcv-pev-prop-ipp} Let $\rho \in C_c(\R^N)$ be a radial function, then for all $u\in L^2(\R^N), \varphi \in C^{\infty}_c(\R^N)$  we have  
$$\iint_{\R^N\times \R^N}\rho(z)[u(x+z)-u(x)]\varphi(x)\,dzdx= \frac{1}{2}\iint_{\R^N\times \R^N}\rho(z)u(x)\Delta_z[\varphi](x)\,dzdx,$$
where $$\Delta_z[\varphi](x):=\varphi(x+z)-2\varphi(x)+\varphi(x-z).$$
\end{proposition}

\begin{proof}
Set $$I:=\iint_{\R^N\times \R^N}\rho(z)[u(x+z)-u(x)]\varphi(x)\,dzdx. $$
By standard change of variable, thanks to the symmetry of $\rho$,  we get
 \begin{align*}
I&=\frac{1}{2} \iint_{\R^N\times \R^N}\rho(z)[u(x+z)-u(x)]\varphi(x) +\frac{1}{2} \iint_{\R^N\times \R^N}\rho(-z)[u(x-z)-u(x)]\varphi(x),\\
 &=\frac{1}{2} \iint_{\R^N\times \R^N}\rho(z)[u(x+z)-u(x)]\varphi(x) +\frac{1}{2} \iint_{\R^N\times \R^N}\rho(z)[u(x)-u(x+z)]\varphi(x+z),\\
 &=-\frac{1}{2} \iint_{\R^N\times \R^N}\rho(z)[u(x+z)-u(x)][\varphi(x+z)-\varphi(x)],\\
 &=-\frac{1}{2} \iint_{\R^N\times \R^N}\rho(z)u(x)[\varphi(x)-\varphi(x-z)] +\frac{1}{2} \iint_{\R^N\times \R^N}\rho(z)u(x)[\varphi(x+z)-\varphi(x)],\\
 &=\frac{1}{2} \iint_{\R^N\times \R^N}\rho(z)u(x)[\varphi(x+z)-2\varphi(x)+\varphi(x-z)].
 \end{align*}
\end{proof}

Consider now $\sigma\le \sigma_2(\O)$ and let $\varphi_{p,\sigma}$ be a positive eigenfunction associated with $\lambda_{p,\sigma}$. That is $\varphi_{p,\sigma}$ satisfies
\begin{equation}\label{bcv-eq-scal-phip1}
\opmem{\varphi_{p,\sigma}}{2,\O}(x)+(a(x)+\lambda_{p,\sigma}) \varphi_{p,\sigma}(x)=0 \quad \text{ for all }\quad x\in \O.
\end{equation}
Let us normalize $\varphi_{p,\sigma}$ by $\nlp{\varphi_{p,\sigma}}{2}{\O}=1$.

 Multiplying \eqref{bcv-eq-scal-phip1} by $\varphi_{p,\sigma}$ and integrating  over $\O$, we get  

\begin{align*}
\frac{D_2(J)}{2}\iint_{\O\times \O}\rho_\sigma(x-y)\frac{|\varphi_{p,\sigma}(y)-\varphi_{p,\sigma}(x)|^2}{|x-y|^2}\,dxdy\le \int_{\O}(a(x)+\lambda_{p,\sigma})\varphi_{p,\sigma}^2(x)\,dx\le C.
 \end{align*} 
Since $a$ and $\lambda_{p,\sigma}$ are bounded independently of $\sigma\le \sigma_2(\O)$, the constant $C$ stands for all $\sigma\le\sigma_2(\O)$. Therefore for any bounded sub-domain $\O'\subset \O$, 
$$\iint_{\O'\times \O'}\rho_{\sigma}(x-y)\frac{(\varphi_{p,\sigma}(y)-\varphi_{p,\sigma}(x))^2}{|x-y|^2}\,dxdy<C. $$

Therefore by the characterisation of Sobolev space in \cite{Ponce2004,Ponce2004a}, for any bounded sub-domain $\O'\subset \O$, along a sequence,  $\varphi_{p,\sigma}\to \varphi$ in $L^2(\O')$. Moreover, by extending $\varphi_{p,\sigma}$ by $0$ outside $\O$, we have $\varphi_{p,\sigma} \in L^2(\R^N)$ and for any $\psi \in C^2_c(\O)$ by Proposition \ref{bcv-pev-prop-ipp} it follows that

\begin{equation} \label{bcv-pev-eq-scal-phip2}
\frac{D_2(J)}{2}\iint_{\O\times \R^N}\frac{\rho_\sigma(z)}{|z|^2}\varphi_{p,\sigma}(x)\Delta_z[\psi](x)\,dxdz = -\int_{\O}(a(x)+\lambda_{p,\sigma}-1+p_\sigma(x))\psi\varphi_{p,\sigma}\,dx. 
\end{equation} 

Recall that $\psi \in C^{\infty}_c(\R^N)$, so there exists $C(\psi)$ and $R(\psi)$ such that for all $x\in \R^N$
$$|\Delta_z[\psi](x)- \transposee{z} (\nabla^{2}\psi(x))z|<C(\psi)|z|^{3}\mathds{1}_{B_{R(\psi)}}(x). $$ 

Therefore, since $\varphi_{p,\sigma}$ is bounded uniformly in $L^2(\O)$,
\begin{equation}\label{bcv-eq-lim-m2-eps-0-3}
 \frac{D_2(J)}{2}\iint_{\O\times\R^N}\frac{\rho_\sigma(z)}{|z|^2}\varphi_{p,\sigma}(x)[\Delta_z[\psi](x)-\transposee{z}(\nabla^{2}\psi(x))z]\,dxdz\le CC(\psi)\int_{\R^N}\rho_n(z)|z| \to 0.
 \end{equation}

On the other hand, $\psi \in C_c^2(\O)$ enforces that  for $\sigma$ small enough  $supp(1-p_\sigma(x))\cap supp(\psi)=\emptyset$. Thus passing to the limit  along a sequence in \eqref{bcv-pev-eq-scal-phip2},thanks to \eqref{bcv-eq-lim-m2-eps-0-3},  we get

 \begin{equation}\label{bcv-eq-clai-no-triv4}
\frac{D_2(J)K_{2,N}}{2}\int_{\O}\varphi(x)\Delta\psi(x)\,dx +\int_{\O}\varphi(x)\psi(x)(a(x)+\lambda_1)\,dx =0.
\end{equation}

\eqref{bcv-eq-clai-no-triv4} being true for any $\psi$, it follows that  $\varphi$ is the smooth positive eigenfunction associated to $\lambda_1$ normalised by $\nlp{\varphi}{2}{\O}=1=\lim_{\sigma \to 0}\nlp{\varphi_{p,\sigma}}{2}{\O}$. The  normalised first eigenfunction being uniquely defined, we get $\varphi=\varphi_1$
 and   $\varphi_{p,\sigma}\to \varphi_1$ in $L_{loc}^2(\O)$ when $\sigma \to 0$.

\fdem

\section*{Acknowledgements}

The research leading to these results has received funding from the European Research Council
under the European Union's Seventh Framework Programme (FP/2007-2013) / ERC Grant
Agreement n°321186 : "Reaction-Diffusion Equations, Propagation and Modelling" held by Henri Berestycki.
J. Coville acknowledges support from the “ANR JCJC” project MODEVOL: ANR-13-JS01-0009 and  the ANR "DEFI" project NONLOCAL: ANR-14-CE25-0013.  
\section*{}
\bibliographystyle{amsplain}
\bibliography{bcv-pev.bib}

\end{document}